\documentclass{article}
\usepackage{arxiv}

\usepackage{amsmath}

\usepackage{graphicx} 
\usepackage{xcolor} 
\usepackage{pifont}
\usepackage{makecell}
\usepackage{amsmath}
\usepackage{amssymb}
\usepackage{multirow}
\usepackage{algorithm}
\usepackage{algorithmic}
\usepackage{amsthm}
\usepackage{mathtools}
\usepackage{hyperref}
\usepackage{bm}

\newtheorem{theorem}{Theorem}
\newtheorem{lemma}{Lemma}
\newtheorem{proposition}{Proposition}

\newcommand{\be}{\mathbf{e}}
\newcommand{\rd}{\mathrm{d}}
\newcommand{\bu}{\mathbf{u}}
\newcommand{\bv}{\mathbf{v}}
\newcommand{\bul}{\mathbf{u}^l}
\newcommand{\buh}{\mathbf{u}^h}
\newcommand{\buth}{\tilde{\mathbf{u}}^h}
\newcommand{\buhat}{\hat{\mathbf{u}}^l}
\newcommand{\bx}{\mathbf{x}}
\newcommand{\by}{\mathbf{y}}
\newcommand{\beps}{\bm{\varepsilon}}
\newcommand{\mN}{\mathcal{N}}
\newcommand{\mI}{\mathbf{I}}
\newcommand{\mWa}{\mathcal{W}_2}
\newcommand{\mz}{\boldsymbol{0}}
\newcommand{\mCR}{\mathcal{R}}

\newcommand{\mE}{\mathbb{E}}
\usepackage{multirow}

\title{Improving Data Fidelity via Diffusion Model-Based Correction and Super-Resolution}

\author{Wuzhe Xu\thanks{Department of Mathematics and statistics, University of Massachusetts Amherst, MA 01002, USA. Email address: wuzhexu@umass.edu} \and 
Yulong Lu\thanks{206 Church St. SE, School of Mathematics, University of Minnesota, Minneapolis, MN 55455, USA.
Email address: yulonglu@umn.edu}
\and
Sifan Wang\thanks{Institution for Foundation of Data Science, Yale University, New Haven, 06520, Email address: sifan.wang@yale.edu}
\and
Tong-Rui Liu\thanks{Department of Aeronautics, Imperial College London, SW7 2AZ, London, UK. Email address:tongrui.liu18@imperial.ac.uk}
}

\date{October 2024}

\begin{document}

\maketitle

\begin{abstract}
    We propose a unified diffusion model-based correction and super-resolution method to enhance the fidelity and resolution of diverse low-quality data through a two-step pipeline. First, the correction step employs a novel enhanced stochastic differential editing technique based on an imbalanced perturbation and denoising process, ensuring robust and effective bias correction at the low-resolution level. The robustness and effectiveness of this approach are validated theoretically and experimentally. Next, the super-resolution step leverages cascaded conditional diffusion models to iteratively refine the corrected data to high-resolution. Numerical experiments on three PDE problems and a climate dataset demonstrate that the proposed method effectively enhances low-fidelity, low-resolution data by correcting numerical errors and noise while simultaneously improving resolution to recover fine-scale structures.
\end{abstract}

\section{Introduction}
High-fidelity, high-resolution (HFHR) data provides the detailed, fine‐scale information necessary to accurately resolve multiscale phenomena, and accurate prediction of phenomena such as extreme weather events, turbulent flows, and long-term climate dynamics. However, obtaining or generating such data is often prohibitively expensive: high-resolution measurements require dense sensor deployments or advanced satellite instrumentation, while high-resolution simulations demand significant computational resources due to the underlying nonlinear, multiscale, and chaotic nature of physical systems. Consequently, only low-fidelity, low-resolution (LFLR) data is typically available in practice, due to both observational and computational limitations. On the observational side, data is often collected on coarse spatial grids and contaminated by noise. For example, in the ERA5 reanalysis dataset \cite{hersbach2020era5}, the spatial resolution of raw satellite measurements varies by sensor and is often very coarse in certain regions. These observations are further degraded by noise arising from instrument limitations, environmental interference, and data processing artifacts \cite{ma2008evaluation,bromwich2004strong}. On the computational side, simulating complex physical systems at high resolution is often infeasible due to the curse of dimensionality. On the other hand, low-resolution solvers, while more efficient, introduces significant numerical errors thus reduce the data fidelity.

One approach to addressing these challenges is to leverage downscaling or super-resolution (SR) techniques, which seek to enhance both the quality and resolution of data starting from low-fidelity, low-resolution (LFLR) inputs. Such techniques, particularly those based on data-driven models, have shown significant promise in computer vision, where deep learning methods have been widely successful in improving image resolution. Examples include models based on convolutional neural network \cite{dong2015image, lim2017enhanced}, neural operators \cite{wei2023super}, Generative Adversarial Network (GAN) \cite{brock2018large, denton2015deep}, and diffusion models \cite{li2022srdiff, saharia2022image}. Recently, efforts have extended the application of these SR models to diverse fields, including climate simulation and prediction \cite{sinha2024effectiveness, watt2024generative, yang2023fourier, rampal2024robust} , medical image enhancement \cite{chan2023super, yoon2024latent, ahmad2022new}, and solving PDEs \cite{li2024physics, Lu2024}. 

While effective, standard SR methods face significant bottlenecks. First, they rely heavily on paired LFLR and HFHR data for training, which are often unavailable or do not  exist in practice. For example, climate systems are chaotic, meaning that small differences in initial conditions can lead to vastly different outcomes over time. As a result, even if a HFHR simulation is initialized with the same conditions as a LFLR simulation, their trajectories will diverge over time. One promising approach that eliminates the need for LFLR and HFHR data pair is to leverage the physics-informed machine learning framework, see \cite{li2024physics,harder2023hard}. However, in many real-world scenarios, such as climate science, the underlying physics is often too complex to model accurately. Even when a good model does exist, it is often prohibitively expensive to exactly enforce the physics. This motivates one primary goal of the paper: to propose a data-driven approach that enhances data fidelity without replying on the underlying physical model.

Another challenge arises from the fact that LFLR data in many physical fields usually contain various sources of biases. This is in contrast to image SR tasks where LR images are usually a straightforward downsampled versions of HR images. For example, in computational fluid dynamics, downsampling a HFHR simulation does not necessarily match the output of the same solver run on a coarse grid. This is because coarse-grid solvers lack the capacity to accurately resolve small-scale processes, such as turbulence, which can accumulate and influence large-scale dynamics. Additionally, when the simulated solution contains sharp transitions, such as shocks or interfaces, that require high resolution to be accurately captured, coarse-grid solvers often fail to resolve these features. This can lead to spurious numerical errors such as artificial diffusion, incorrect wave speeds, or oscillations near discontinuities. Furthermore, in many experimental settings, LFLR data is often contaminated by various types of noise, as previously discussed\cite{ma2008evaluation,bromwich2004strong}. Therefore, it is crucial to correct the bias prior to  the super-resolution. 

Although deep learning models \cite{levin2012patch, mao2016image, zhang2017beyond, zhang2018ffdnet} have shown promising performance in debiasing errors in LFLR data, they typically rely on training datasets that contain a specific type of bias. As a result, these models often struggle to generalize across diverse biases arising from different error sources. This limitation highlights the need for a more versatile debiasing approach that can effectively address biases in an error-source-agnostic manner.

The goal of this paper is to address the aforementioned challenges by developing a unified computational method capable of enhancing various types of LFLR data without requiring knowledge of the underlying physics.




\subsection{Problem Description}
Let $\bul$ and $\buh$ denote the LFLR data and HFHR data respectively. Beyond the resolution difference, LFLR data often contains biases that must be corrected at the low-resolution level. To relate the two, we introduce a downsampling operator $\mCR$ that maps the HFHR data $\buh$ to its low-resolution counterpart $\buth = \mCR \buh$, referred to as high-fidelity low-resolution (HFLR) data, which matches the resolution of the LFLR data $\bul$. The discrepancy between the LFLR data $\bul$ and the HFLR data $\buth$ defines the bias $\be$, yielding the relation $\bul = \buth + \be$. These biases may arise from different sources, such as running different low-resolution models in scientific computing problems or data being polluted by various types of noise. To account for multiple types of biases, we use $\be_i$ to denote $i$-th type of bias in LFLR data, leading to $\bul_i = \buth + \be_i, ~i=1, \cdots, n$. Based on this setting, the problem of interest can be addressed using a two-step pipeline:
\begin{enumerate}
    \item \textit{Correction} step: Remove the biases $\be_i$ from $\bul_i$ to recover the corresponding $\buth$, using a unified model for all $i = 1, \dots, n$.
    \item \textit{Super-resolution} step: Upscale $\buth$ to reconstruct $\buh$.
\end{enumerate}
For simplicity, we slightly abuse notation by using $\buth$ to also denote the corrected HFLR approximation in the context above. This two-step pipeline, along with the relationships among the LFLR $\bul$ data, HFLR data $\buth$ and HFHR data $\buh$, are illustrated in Figure~\ref{fig:demo_u}. 
\begin{figure}[h]
    \centering
    \includegraphics[height=0.35\linewidth, width=0.55\linewidth]{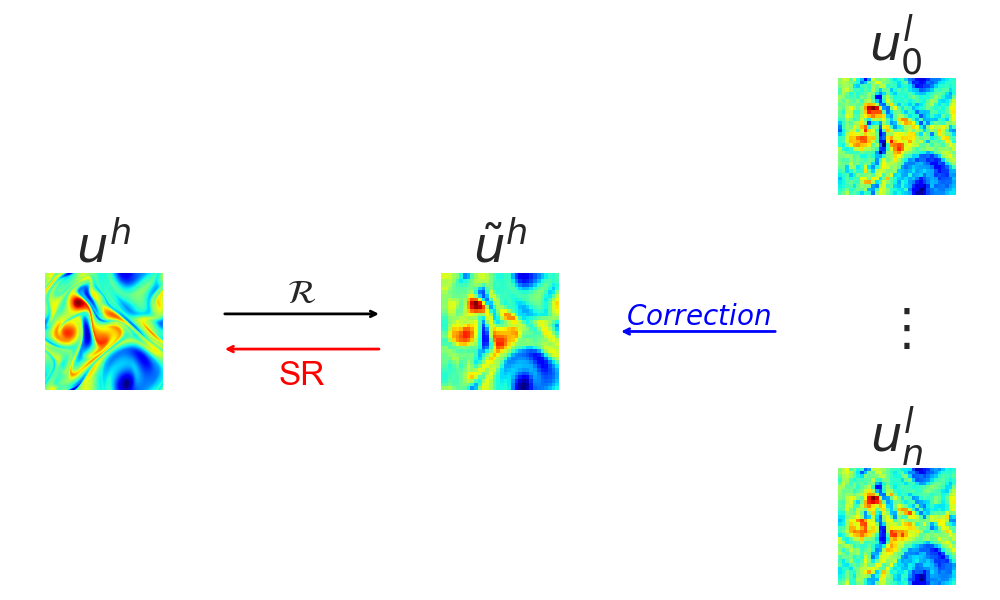}
    \caption{Diagram of the two-step pipeline. The Correction step (blue arrow) removes various biases from LFLR data $\bul_i, i=1, \cdots, n$ to recover HFLR data $\buth$. The restriction operator $\mathcal{R}$ (black arrow) maps the HFHR data $\buh$ to its HFLR version $\buth$ and the Super-Resolution step (red arrow) performs a pseudo-inverse mapping to enhance the resolution of HFLR data $\buth$ and reconstruct the HFHR data $\buh$.}
    \label{fig:demo_u}
\end{figure}

\subsection{Related Works}

\textbf{Neural Operators.} Neural operators are designed to approximate mappings between infinite-dimensional function spaces. Notable architectures includes Fourier Neural Operators (FNO) \cite{li2021fourier, guibas2021adaptive, wen2022u, fanaskov2023spectral, you2022learning}, Deep Operator Networks (DeepONet) \cite{lu2021learning, wang2022improved, wang2021learning, wang2023long, zhu2023fourier, goswami2022physics}, and Transformed-based networks \cite{wang2024bridging, hao2023gnot, li2022transformer,mccabe2023multiple}. Since the output of a neural operator is a function, it naturally supports zero-shot super-resolution. This feature suggests its potential to bridge the resolution gap without relying on high-fidelity training data. However, as highlighted in \cite{li2024physics}, training on low-resolution data and predicting on a high-resolution grid inevitably fails to capture the fine structures in the data. To remedy this issue, the authors enforce underlying physics by introducing an additional physics-informed loss. However, in our setting, we do not have the access to the underlying physics. An alternative approach, such as \cite{wei2023super,yang2023fourier}, that using neural operators for downscaling is to approximate the pseudo-inverse operator that maps LFLR data to their HFHR counterparts. In this case, neural operators require paired datasets $\{ \bul, \buh \}$ to train a pseudo-inverse mapping $\mathcal{G}: \bul \rightarrow \buh$. However, their reliance on  paired data makes them unsuitable for our problem.

\textbf{Unpaired Domain Translation} The task of correcting LFLR data to align with HFLR data in the Correction step is an instance of the unpaired domain translation task, which translates data from a source domain (LFLR) to a target domain (HFLR). Many data-driven models have been proposed to address this domain translation problem, particularly in image-to-image translation tasks. For example, CycleGAN \cite{zhu2017unpaired}, a GAN-based framework for unpaired image-to-image translation, achieves this by introducing a novel loss of cycle consistency that captures the underlying correspondence between domains. Another approach is the Optimal Transport (OT)-based method, such as Neural Optimal Transport (NOT) \cite{korotin2022neural}, which parameterizes the optimal transport maps using neural networks, enabling efficient unpaired image-to-image translation. Additionally, the Unpaired Neural Schrödinger Bridge (UNSB) \cite{kim2023unpaired} has been proposed to learn a SDE for translating between two distributions, offering another effective method for unpaired domain translation. 

While these frameworks address the domain translation task using unpaired datasets, they are designed for translation between a specific pair of the source and target domains and cannot be directly generalized to other domain pairs. Here, we refer to this as a "domain pair", which should not be confused with "paired datasets" that emphasizes the one-to-one correspondence between individual data points from the source and target domains. Applying these frameworks to multiple domains translation requires a quadratic number of models relative to the number of domains, limiting their scalability. To resolve this, a Dual Diffusion Implicit Bridges (DDIB) \cite{su2022dual} have been proposed. DDIB utilizes two chained diffusion models \cite{song2020score,ho2020denoising}, one for each domain, and bridges the two domains through a shared latent space, circumventing the need for paired data. Furthermore, because these two diffusion models are trained independently on their respective domains, they can be reused for translation tasks involving unseen domains. However, this approach remains task-specific, as it requires empirical datasets from both domains for training.

\textbf{Diffusion-based Correction.} 
The SDEdit \cite{meng2021sdedit} provides a promising solution to the challenges of unpaired training data and task-specific limitations discussed earlier. Unlike DDIB, which requires training separate diffusion models for each domain and bridges them through a shared latent space, SDEdit eliminates the need for source domain data during training. The primary idea is to perturb the data with a appropriate level of noise, such that inherent biases $\be$ become negligible compared to the injected noise. A diffusion model trained to generate target domain data is then applied for denoising, removing the combined noise (injected noise and inherent bias) to produce a corrected data that aligns with the target domain. This makes SDEdit particularly well-suited for multi-task bias correction,  and it has thus been widely used in various scenarios for removing biases from data, a process commonly referred to as data purification \cite{nie2022diffusion, yoon2021adversarial}. 

However, the direct application often lack accuracy and robustness to various types and scales of bias. To address these limitations, enhanced purification diffusion models \cite{wang2022guided, li2024adbm} have been proposed. However, these methods are either computationally expensive or rely on prior knowledge of the bias type and magnitude during training. Unfortunately, such prior knowledge is rarely available in practice. For example, in scientific computing problems, biases may depend on the distribution of the high-fidelity data, which is often unknown. In this paper, we propose a novel imbalanced perturbing and denoising method for SDEdit to correct biases in LFLR data. This method significantly enhances correction performance by improving robustness and accuracy, as validated through both theoretical analysis (Section~\ref{sec:dc}) and experimental results (Section~\ref{sec:numerical}).

\textbf{Direct Super-resolution} Although many deep learning frameworks have proven effective for Super-Resolution tasks, they often suffer from various limitations. For instance, the Normalization Flow models \cite{groenke2020climalign, winkler2024climate} are computationally prohibitive for high-resolution image generation and often produce suboptimal sample quality, and GAN \cite{li2024generative, harris2022generative} require carefully designed regularization and optimization techniques to address instability during training. To address these limitations, we adopt Image Super-Resolution via Iterative Refinement (SR3) \cite{saharia2022image}, a diffusion model-based SR framework known for its training stability, computational efficiency, and ability to generate high-quality HR images through iterative refinement. As discussed by the authors, the SR3 can be applied in a cascaded manner, dividing a high magnification super-resolution task into a sequence of smaller super-resolution tasks. This cascading approach, referred to as cascade SR3, is more efficient, requiring fewer sampling steps to achieve comparable high-resolution quality, and thereby is the model adopted for the SR step in this work.

\subsection{Most Relevant Works}
Two works most relevant to ours, though based on different approaches and for different goals, are \cite{wan2023debias} and \cite{xu2025diffusion}. Both propose probabilistic super-resolution methods for unpaired data by decomposing the downscaling task into two stages. First, they apply a debiasing step at the low-resolution (LR) level, using an OT-based model in \cite{wan2023debias} and an enhanced DDIB model in \cite{xu2025diffusion}. This is followed by an upsampling step to enhance the resolution of the debiased data. Despite these similarities, our approach differs from both in two significant ways. First, their methods focus on recovering statistical properties, whereas our approach emphasizes improving data fidelity by explicitly removing biases. Second, their models are task-specific, requiring a separate model for each type of bias and thus lacking versatility.



\subsection{Contributions and Applications}
The contribution of this paper are two-fold:
\begin{itemize}
    \item \textbf{Versatile Fidelity Improvement Model}: A novel purely data-driven fidelity improvement method, diffusion model-based correction and super-resolution (DCSR), is proposed. This approach trains solely on HFHR reference data and can then be effectively applied to enhance the fidelity of various types of LFLR data.

    \item \textbf{Enhanced SDEdit with Imbalanced Perturbing and Denoising (IPD)}: We propose an enhanced SDEdit that incorporates a novel Imbalanced Perturbing and Denoising process for the correction step. This method enables robust correction of various biases, including numerical errors and noise, and is validated both theoretically and experimentally.
    
\end{itemize}

\subsection{Structure of the Paper}
Section~\ref{sec:dm_intro} reviews the background of diffusion models. Section~\ref{sec:method} introduces the proposed DCSR method, which enhances the fidelity of LFLR data through two main steps. In Subsection~\ref{sec:dc}, we develop an enhanced SDEdit based on a novel imbalanced perturbing and denoising process for data correction at the LR level, along with the theoretical analysis demonstrating its robustness and effectiveness. Next in Subsection~\ref{sec:sr} describe the SR3 model, which upsamples the corrected LR data to its HR counterpart. The full pipeline is illustrated in Subsection~\ref{sec:dcsr}. Finally, Section~\ref{sec:numerical} presents numerical experiments on three PDE benchmarks and one climate dataset to demonstrate the effectiveness of the proposed method.

\section{Background of Diffusion Models}\label{sec:dm_intro}
Score-based diffusion models \cite{song2020score, yang2023diffusion} are designed to learn a target distribution $P(\bx)$ from a set of training samples, enabling the generation of new data drawn from $ P(\bx)$. In this work, an unconditional diffusion model is utilized in the correction step to refine LR data. For the SR step, we adopt a conditional diffusion model to learn the conditional distribution  $P(\bx|\by)$, where $\bx$ denotes the high-resolution data and $\by$ the corresponding low-resolution input. Under this formulation, the SR problem naturally becomes a conditional sampling task. In the following, we briefly review the mathematical formulation of diffusion models in the unconditional setting; the extension to the conditional setting is straightforward.

Diffusion models involve a two-stage process. In the forward process, the data is gradually corrupted with noise according to a user-specified, transforming the target data distribution $P(\bx)$ into a Gaussian distribution. Two standard choices for the SDE are the Variance Exploding SDE (VE SDE) and the Variance Preserving SDE (VP SDE)\cite{song2020denoising}. Based on our numerical results, they perform similarly for our problem and we adopt the VE SDE in this paper for simplicity. The VE SDE is defined as:
\begin{equation}\label{eqn:diff}
\mathrm{d}\bx(t) =  \sqrt{\frac{\rd [\sigma^2(t)]}{\rd t}} \, \mathrm{d}\mathbf{w}    
\end{equation}
where $\mathrm{d}\mathbf{w}$ is a standard Wiener process, $\sigma(t)$ is a time-dependent noise scheduling function, and the initial condition is $\bx(0):=\bx \sim P(\bx)$. Over time, the distribution evolves to $\bx(t) \sim \mathcal{N}(\bx, \sigma^2(t) \mI)$. By selecting an appropriate scheduling function $\sigma(t)$ such that the variance at $t=1$ is much larger than the magnitude of $\bx$, the marginal distribution at $t=1$ can be approximated as an isotropic Gaussian distribution $\mathcal{N}(\bx, \sigma^2(t) \mI) \approx \mathcal{N}(\mz, \sigma^2(t) \mI)$. The reverse process is described by the corresponding reverse time SDE that progressively transforming the isotropic Gaussian distribution back into the original data distribution. The reverse-time SDE is given by:
\begin{equation}\label{eqn:reverse}
\mathrm{d}\mathbf{x}(t) = \left[ -  \frac{\rd [\sigma^2(t)]}{\rd t} \nabla_{\mathbf{x}} \log p_t(\mathbf{x})\right] \mathrm{d}t + \sqrt{\frac{\rd [\sigma^2(t)]}{\rd t}} \, \mathrm{d}\bar{\mathbf{w}},
\end{equation}
where $\mathrm{d}\mathbf{\bar{w}}$ represents a reverse-time Wiener process. Moreover, Song et al \cite{song2020score} prove the existence of the probability flow ODE, which shares the same marginal distributions as the reverse-time SDE:
\begin{equation}\label{eqn:pfode_ori}
\frac{\rd \bx}{\rd t} = - \frac{1}{2} \frac{\rd [\sigma^2(t)]}{\rd t} \nabla_{\mathbf{x}} \log p_t(\mathbf{x}),
\end{equation}
In practice, the noise-perturbed score function $\nabla_{\mathbf{x}} \log p_t(\mathbf{x})$ is approximated using a time-dependent neural network $S_{\theta}(\bx_t, t)$ trained via a score-matching objective:
\begin{equation}\label{eqn:loss}
    \theta^* \in \arg \min_{\theta} L(\theta):= \mathbb{E}_{t \sim U(0, 1), \bx \sim p(\bx), \beps \sim \mN(\mz, \mI)} [\| \sigma(t) S_{\theta}(\bx(t), t) + \beps \|_2^2].
\end{equation}
With a well-trained $S_{\theta}(\bx_t, t)$, a new data can be generated using either the PF ODE \eqref{eqn:pfode_ori} or reverse-time SDE \eqref{eqn:reverse}. In the former approach, a new samples can be generated by solving the back PF ODE \eqref{eqn:pfode_ori} using any numerical ODE solver. We use notation $\text{ODEsolve}$ to present the solution of ODE driven by velocity field $-\frac{1}{2} \frac{\rd [\sigma^2(t)]}{\rd t} v(\bx(t), t)$, starting from initial condition $\bx(t_1)$ at $t_1$ and evolving to $t_2$. This solution is formally denoted as:
\begin{equation}\label{eqn:pfode_notation}
\text{ODEsolve}(\bx(t_1), t_1, t_2; v):=\bx(t_2) = \bx(t_1) + \int_{t_1}^{t_2} - \frac{1}{2} \frac{\rd [\sigma^2(t)]}{\rd t} v(\bx(t), t) \rd t
\end{equation}
Under this setting, a new samples generated from the PF ODE can be represented as:
\begin{equation}\label{eqn:pfode}
\bx(0) = \text{ODEsolve}(\bx(1), 1, 0; S_{\theta}), \quad \bx(1)\sim \mN(\mz, \sigma^2(t) \mI).
\end{equation}
We adopt this approach for the \textit{correction} step due to it's efficiency. In general, generating data using reverse-time SDEs produces higher quality results compared to using an ODE solver but is less efficient due to the requirement of a large number of sampling steps. Therefore, for the SR step, where recovering small-scale details is critical, we use the reverse-time SDE \eqref{eqn:reverse} for sampling. This is typically implemented via the Euler-Maruyama method, starting with an initial condition $\bx(1) \sim \mN(\mz, \sigma^2(1) \mI)$, and evolving over time interval $t=1$ to $t=0$ with a uniform time discretization step $\Delta t = \frac{1}{N_T}$. At each step, the data is updated from $t$ to $t-\Delta t$ using the rule:
\begin{equation}\label{eqn:EM}
\bx(t-\Delta t) = \bx(t) + (\sigma^2(t-\Delta t) - \sigma^2(t)) S_{\theta}(\bx(t), t) + \sqrt{\sigma^2(t) - \sigma^2(t-\Delta t)} \boldsymbol{z}, \quad \boldsymbol{z} \sim \mN(\mz, \mI).
\end{equation}
This iterative process continues until $t=0$, producing a new sample $\bx(0)$. While various fast sampling methods, such as \cite{song2020denoising,zhang2022fast, lu2022dpm}, have been proposed, they are beyond the scope of this paper and are mentioned here for completeness.

\section{Methodology}\label{sec:method}
\subsection{Diffusion-Based Correction}\label{sec:dc}

\noindent \textbf{SDEdit}. 
SDEdit \cite{meng2021sdedit} was originally proposed as a diffusion-based image editing technique. In particular, it was demonstrated to efficiently transform stroke-based images into realistic versions while preserving their overall structure. This idea was later extended for various other tasks, including the data purification that defends against adversarial perturbations \cite{nie2022diffusion, wang2022guided}. Although these applications use different terminology, their objective aligns closely with the \textit{correction} step discussed in this paper: removing the bias term $e$ from LFLR $\bul = \buth +e$ to recover HFLR $\buth$. This process consists of two stages. In the first stage, a LFLR data $\bul$ is perturbed by running the forward-time SDE \eqref{eqn:diff} with time $t$, and the perturbed data is denoted as $\bul(t)=\buth + \be + \sigma(t) \beps,~ \beps\sim \mN(\mz, \mI)$. In the second stage, a diffusion model trained on the HFLR dataset $\{ \buth \}$ is used for denoising, which removes the combined noise (the existing bias $e$ and the added perturbation) by running the backward PF ODE \eqref{eqn:pfode_notation} for same time $t$ to generate a correction that approximates the HFLR $\buth$. 

This unconditional diffusion model is denoted by $S_{\theta}(\buth(t), t)$, where $\theta$ is the trainable weights, and it is trained by minimizing the following Loss function
\begin{equation}\label{eqn:loss_dc} L(\theta):= \mathbb{E}_{t \sim U(0, 1), \buth \sim P(\buth), \beps \sim \mN(\mz, \mI)} [\| \sigma(t) S_{\theta}(\buth(t), t) + \beps \|_2^2].
\end{equation}
The training procedure is detailed in Algorithm~\ref{alg:utilde}.
\begin{algorithm}[H]
\caption{Unconditional Diffusion Model for HFLR data $\buth$}
\label{alg:utilde}
\begin{algorithmic}[1]
\REQUIRE Training datasets $\mathcal{T} = \{ \tilde{u}_i^h \}_{i=1}^N$, noise scheduling function $\sigma(t)$, batch size $B$ and max iteration $Iter$
\STATE Initialize $k=0$
\WHILE{$k<Iter$}
    \STATE Sample $\{ \tilde{u}_j^h \}_{j=1}^B \sim \mathcal{T}$ 
    \STATE $t \sim U[0, 1]$
    \STATE $\beps_j \sim \mN(\mz, \mI)$ for $j=1, \cdots, B$
    \STATE Compute $\tilde{u}_j^h(t) = \tilde{u}_j^h + \sigma(t) \beps_j$
    \STATE  Update $\theta$ using Adam to minimize the empirical loss \eqref{eqn:loss_dc}
    \STATE $k \gets k+1$
\ENDWHILE
\RETURN Diffusion Model $S_{\theta}(\buth(t), t)$
\end{algorithmic}
\end{algorithm}

Existing diffusion-based data correction approaches use balanced perturbing and denoising (BPD), where the forward perturbation and the backward denoising processes are evolved for the same amount of time $t$. However, BPD faces a fundamental challenge in selecting the optimal hyperparameter $t$. A sufficiently large $t$ is necessary to ensure the injected noise $\sigma(t)\beps$ dominates the existing bias $\be$. Yet, excessive $t$ values can deteriorate the large-scale structure, causing the denoising process to generate random samples from the HFLR distribution $P(\buth)$ rather than preserving the original structural features.
The robustness of BPD relies on a good choice of the parameter $t$. 

To address this issue, several improvements have been proposed. For instance, in \cite{wang2022guided}, the authors propose to alleviate this issue by first repeating the correction step multiple times with a relatively small $t$ in each time, and then employing a guided diffusion model that uses the LFLR data as a guidance to ensure the correction does not deviate significantly from the LFLR, thereby preserving the large-scale structure. However, this approach is computationally inefficient due to the repeated correction steps and the extra guidance that requires gradient computations. In \cite{li2024adbm}, the authors directly constructed a reverse bridge to map LFLR data to its HFLR counterpart. While effective, this method relies on the prior knowledge of the bias within the LFLR data for training, which is unavailable in our setting.


\textbf{Imbalanced Perturbation and Denoising}. To address the lack of robustness in standard BPD, we propose the IPD, which improves upon BPD by using two independent times, $t_1$ and $t_2$, for the perturbation and denoising steps, respectively, instead of a single shared time $t$. We denote the LFLR data distribution and HFLR data distribution by $p(\bul)$ and  $p(\buth)$ respectively. Starting with the initial condition $\bx(0) = \bul \sim p(\bul)$, running the forward-time SDE \eqref{eqn:diff} for a time $t_1$ produces a perturbed solution $\bul(t_1)$, distributed as $p(\bul(t_1))$. Similarly, starting with initial condition $\bx(0) = \buth \sim p(\buth)$, running the same forward-time SDE \eqref{eqn:diff} for a time $t_2$ produces a solution $\buth(t_2)$, distributed as $p(\buth(t_2))$. During the perturbation process, $p(\bul(t_1))$ and $p(\buth(t_2))$ gradually mix and align closely as $t_1, t_2$ increase, and eventually converging to the same isotropic Gaussian distribution $\mN(\mz, \sigma^2(1) \mI)$ as $t_1,t_2 \rightarrow 1$. This implies that there exist $t_1$ and $t_2$ such that the perturbed data $\bul(t_1)$ can be treated as a sample from the distribution $p(\buth(t_2))$. Building on this, we design the IPD mechanism to decouple the perturbation and denoising steps by independently choosing $t_1$ and $t_2$:
\begin{itemize}
    \item \textbf{Perturbation}. Perturb the LFLR data $\bul$ by running forward time SDE for time $t_1$, resulting in a perturbed data
    \begin{equation}\label{eqn:perturbing}
        \bul(t_1)=\bul + \sigma(t_1) \beps, ~\beps \sim \mN(\mz, \mI).
    \end{equation}
    \item \textbf{Denoising.} Run the backward PF ODE \eqref{eqn:pfode_notation} from $t=t_2$ to $t=0$, with initial condition $\bx(t_2)=\bul(t_1)$ and the score function $S_{\theta}(\buth(t), t)$ as the velocity field. We denote by the resulting correction by 
    \begin{equation}\label{eqn:recon}
\buhat(t_1, t_2) = \text{ODEsolve} (\bul(t_1), t_2, 0; S_{\theta}).
\end{equation}
\end{itemize}
The illustration of the IPD is provided in Figure~\ref{fig:diagram}.

\begin{figure}[H]
    \centering
    \includegraphics[width=0.8\linewidth]{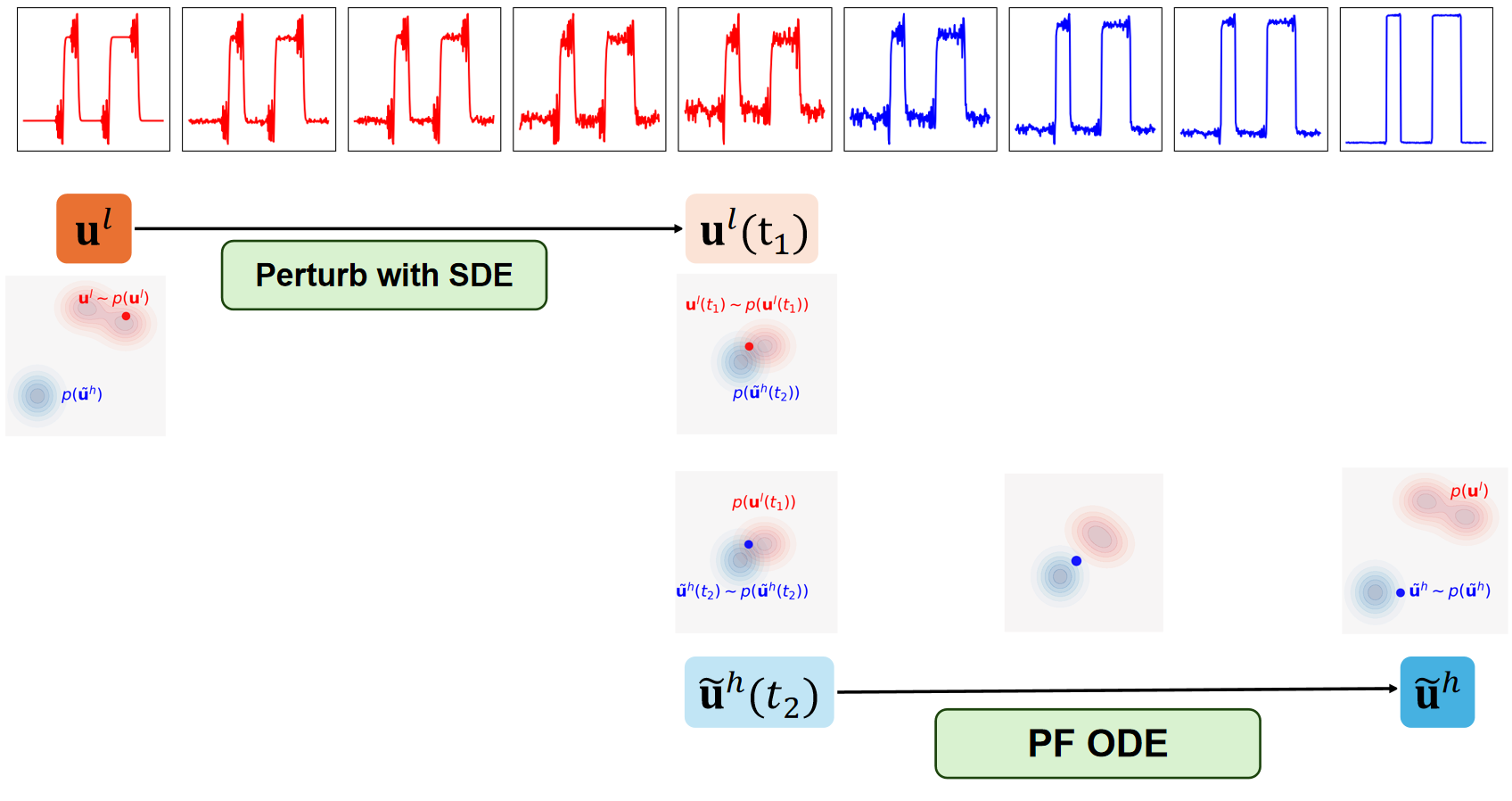}
    \caption{Diagram of the imbalanced perturbing and denoising (IPD) process. The process begins with perturbing the LFLR data $\bul$ using forward time SDE \eqref{eqn:diff} to produce perturbed data $\bul(t_1)$, whose distribution $p(\bul(t_1))$ aligns with the perturbed HFLR distribution $p(\buth(t_2))$. As a result, the perturbed data $\bul(t_1)$ can be treated as a sample from distribution $p(\buth(t_2))$. Subsequently, the Probability Flow (PF) ODE maps take this perturbed data to the HFLR data $\buth$. The first row illustrates how an LFLR data point is corrected to produce the HFLR data, while the bottom two rows depict the evolution of distributions throughout the process.}
    \label{fig:diagram}
\end{figure}

While the IPD mechanism enables flexible and decoupled perturbation and denoising, the choice of time parameters $t_1, t_2$ plays a critical role in its effectiveness. These parameters must be sufficiently large to ensure that $P(\bul(t_1))$ and $P(\buth(t_2))$ align closely. However, excessively large values of $t_1, t_2$ may result in an over-perturbed $\bul(t_1)$, potentially destroying the large-scale structure in the LFLR data $\bul$ and degrading the correction quality. This trade-off is theoretically characterized by Theorem~\ref{thm}, whose proof is provided in the Appendix~\ref{sec:proof}.

\begin{theorem}\label{thm}
    For a HFLR data $\buth \sim p(\buth)$, let $\bul = \buth + \be$ be a LFLR data, where $\be$ is the bias. Assume that the approximate score function $S_{\theta}(\buth(t), t)$ is $L_s$-Lipschitz continuous in $\buth(t)$ for $t\in [0, 1]$ and $0<t_1<t_2<1$. If the empirical training loss defined in equation \eqref{eqn:loss_dc} satisfies $L(\theta)<\delta$, for all $\lambda \in (0,1)$, with probability at least $1-\lambda$, the $L^2$ distance between the target HFLR data $\buth$ and the correction $\buhat(t_1, t_2)$ is bounded by
    \begin{align*}
        \| \buhat(t_1, t_2) - \buth \|^2_2 \leq e^{2L_s t_2} [\| \be \|_2^2 + \sigma^2(t_2)\delta + (\sigma^2(t_1) + \sigma^2(t_2)) C_{\lambda}].
    \end{align*}
    where $C_\lambda = d + 2 \sqrt{d \log\frac{1}{\lambda}} + 2\log\frac{1}{\lambda}$ and $d$ is the dimension of $\buth$. 
\end{theorem}
As indicated by the above theorem, and noting that $\sigma(t)$ is a monotonically increasing function, minimizing the correction error requires keeping both $t_1$ and $t_2$ relatively small. Consequently, there exists a fundamental trade-off: larger values of $t_1$ and $t_2$ improve distribution alignment and bias correction, while smaller values better preserve the large-scale structure of the LFLR data. This trade-off is inherent in BPD, which uses a single time parameter $t_1 = t_2 = t$ and thus cannot balance these competing objectives. In contrast, IPD decouples the perturbation and denoising times, providing greater flexibility to mitigate this trade-off through the independent and appropriate selection of $t_1$ and $t_2$. 

\textbf{Theoretical Illustration with Gaussian Noise Example}. To analytically illustrate the advantage of IPD over BPD, we consider a toy example where the bias in the LFLR data is sampled from an isotrophic Gaussian distribution, $\be \sim \mN(\mz, \gamma^2 \mI)$, and we present the following proposition to highlight the benefit of using IPD.
\begin{proposition}\label{pro1}
    For a HFLR data $\buth \sim p(\buth)$, let $\bul = \buth + e$ be a LFLR data, where the bias term $e \sim p(e)=\mN(\mz, \gamma^2 \mI)$. Assuming that the approximated score function $S_{\theta}(\buth(t), t)$ is $L_s$-Lipschitz continuous in $\buth(t)$ for $t\in [0, 1]$. If the empirical training loss defined in equation \eqref{eqn:loss_dc} satisfies $L(\theta)<\delta$, then for any $0<t_1<t_2<1$ such that $\sigma^2(t_2) =\sigma^2(t_1) + \gamma^2$, the expected $L^2$ distance between $\buth$ and the correction $\buhat(t_1, t_2)$ is bounded by:
    \begin{equation}
        \mE_{\buth \sim p(\buth), \varepsilon \sim \mN(\mz, \mI), \be \sim p(\be)}(\| \buhat(t_1, t_2) - \buth \|^2_2) \leq e^{2L_st_2} \sigma^2(t_2) \delta.
    \end{equation}
\end{proposition}
As previously noted, BPD corresponds to the special case of IPD with $t_1 = t_2 = t$, where the goal is to select a sufficiently large $t$ such that $\sigma^2(t)>>\gamma^2$, allowing the injected noise to dominate and effectively ``screen out" the bias.  According to Proposition~\ref{pro1}, the upper bound for the expected correction error is $e^{2L_st_2} \sigma^2(t_2) \delta$, where $\sigma^2(t)>>\gamma^2$. In contrast, by selecting $t_1=0$ and $t_2=\sigma^{-1}(\gamma)$ using IPD, the Proposition~\ref{pro1} implies a much smaller upper bound, $e^{2L_s \sigma^{-1}(\gamma)} \gamma^2 \delta$. Moreover, when the training loss approaches zero, this bound vanishes, indicating that Gaussian-type bias can be completely removed. While the benefits of IPD is only shown rigorously in the simple setting where the biases are Gaussian white noises, our numerical results in Section \ref{sec:numerical} demonstrate the superiority of IPD in a wide range of biases, such as numerical errors and various kinds of noises.

\textbf{Selection of $t_1$ and $t_2$.}
After demonstrating the theoretical advantage of IPD in the Gaussian noise setting, we now address the practical selection of $t_1$ and $t_2$ for robust correction in more general bias scenarios. Our key idea is to search for the smallest pair $t_1$ and $t_2$ such that $p(\bul(t_1))$ align closely with $p(\buth(t_2))$ under certain metrics. There are two critical factors in this procedure: first, the choice of the metric $\mathcal{M}$ used to evaluate distributional alignment; second, to preserve large-scale structures, both $t_1$ and $t_2$ are constrained to lie within the interval $[0, T_e]$ for a suitably chosen $T_e$. Both the metric $\mathcal{M}$ and the terminal searching time $T_e$ are selected empirically, and the criteria for choosing $\mathcal{M}$ will be discussed in detail in Section~\ref{sec:1d_adv}. The complete procedure for selecting $t_1$ and $t_2$ is detailed in Algorithm~\ref{alg:t1t2}. For completeness, we also include Algorithm~\ref{alg:t} in the Appendix~\ref{sec:supp_alg}, which provides an analogous approach for determining the optimal $t$ in the BPD setting.

\begin{algorithm}
\caption{Selection of $t_1$ and $t_2$}
\label{alg:t1t2}
\begin{algorithmic}[1]
\REQUIRE A HFLR dataset $\{ \buth_i \}_{i=1}^N$ and a LFLR datasets $\{ \bul_j \}_{j=1}^M$, terminal searching time $T_e$, number of steps $N_{t_1}$ and $N_{t_2}$, two scalars $1<c_1<c_2$, and a metric $\mathcal{M}$.
\STATE Initialize $d_{\min} \gets \infty$
\STATE Initialize $t^{*}_1 \gets 0$, $t^{*}_2 \gets 0$
\FOR{$p \gets 0$ to $N_{t_1}-1$}
    \STATE $t_1 \gets p \cdot \frac{T_e}{N_{t_1}-1}$
    \FOR{$q \gets 0$ to $N_{t_2}-1$}
    \STATE $t_2 \gets c_1 t_1 + q \cdot \frac{c_2-c_1}{N_{t_2}-1} t_1$
    \STATE Obtained perturbed $\{ \bul_j (t_1) \}_{j=1}^M$ via SDE \eqref{eqn:diff} for time $t_1$
    \STATE Obtained perturbed $\{ \buth_i (t_2) \}_{i=1}^N$ via SDE \eqref{eqn:diff} for time $t_2$
    \STATE Compute $d_i = \mathcal{M}(\{ \bul_j (t_1) \}_{j=1}^M, \{ \buth_i (t_2) \}_{i=1}^N)$
    \IF{$d_i < d_{\min}$}
        \STATE $d_{\min} \gets d_i$
        \STATE $t^{*}_1 \gets t_1$, $t^{*}_2 \gets t_2$
    \ENDIF
    \ENDFOR
\ENDFOR
\RETURN Optimal $t^{*}_1$ and $t^{*}_2$
\end{algorithmic}
\end{algorithm}

\textbf{Enhanced SDEdit with IPD.} The full procedure for correcting the LFLR dataset using Enhanced SDEdit with IPD is summarized as follows. Given a HFLR dataset $\{ \buth_i \}_{i=1}^N$, an unconditional diffusion model is first pre-trained. The optimal perturbation and denoising times $t_1^*$ and $t_2^*$ are then selected using Algorithm~\ref{alg:t1t2}. Subsequently, the LFLR dataset $\{ \bul_j \}_{j=1}^M$ is corrected by applying perturbation and denoising via the SDE~\eqref{eqn:diff} using $t_1^*$ and $t_2^*$ respectively, resulting the correted dataset $\{ \hat{\mathbf{u}}^{l,*}_j\}_{j=1}^M$. The complete procedure is detailed in Algorithm~\ref{alg:dc}.
\begin{algorithm}[h]
\caption{Diffusion-based Correction}
\label{alg:dc}
\begin{algorithmic}[1]
\REQUIRE A HFLR dataset $\{ \buth_i \}_{i=1}^N$ and a LFLR datasets $\{ \bul_j \}_{j=1}^M$, terminal searching time $T_e$, number of steps $N_{t_1}$ and $N_{t_2}$, two scalars $1<c_1<c_2$, a metric $\mathcal{M}$ and a well trained model $S_{\theta}$.
\STATE Obtain $t^{*}_1$ and $t^{*}_2$ from Algorithm~\ref{alg:t1t2}
\FOR{$j \gets 1$ to $M$}
\STATE Obtain perturbed $\bul_j (t^{*}_1)$ via SDE \eqref{eqn:diff} for time $t^*_1$
\STATE Obtain correction $\hat{\mathbf{u}}^{l,*}_j := \buhat_j (t^{*}_1, t^{*}_2) = \text{ODEsolve} (\bul_j (t^{*}_1), t^{*}_2, 0; S_{\theta})$.
\ENDFOR
\RETURN Corrected dataset $\{ \hat{\mathbf{u}}^{l,*}_j\}_{j=1}^M$
\end{algorithmic}
\end{algorithm}

\subsection{Super-resolution via Iterative Refinement}\label{sec:sr}
In the SR step, we employ the cascaded SR3 model \cite{saharia2022image} to refine the resolution of the corrected LF data. Since the restriction operator $\mathcal{R}$ is a user-specified and known, a paired dataset of HFLR and HFHR data, $\{ \buth_i, \bu_i^h \}_{i=1}^N$, can be constructed from a given HFHR dataset $\{\bu_i^h \}_{i=1}^N$ by applying $\mathcal{R}$ to each HFHR sample. This paired dataset enables the training of a conditional diffusion model $S_{\xi}(\buh(t), \buth, t)$ to approximate the conditional distribution $P(\buh | \buth)$ by minimizing the following loss function:
\begin{equation}\label{eqn:loss_sr} 
L(\xi):= \mathbb{E}_{t \sim U(0, 1), \buh \sim P(\buh), \beps \sim \mN(\mz, \mI)} [\| \sigma(t) S_{\xi}(\buh(t), \buth, t) + \beps \|_2^2],
\end{equation}
The complete training procedure for this conditional diffusion model $S_{\xi}(\buh(t), \buth, t)$ is a straightforward extension of the unconditional case and the details are presented in Algorithm~\ref{alg:uh} in the Appendix~\ref{sec:supp_alg}. Once the model is trained, a HFHR data $\buh$ corresponding to the given HFLR data $\buth$ can be generated using the EM method adapted to the conditional setting. The generating process starts with $\buh(1) \sim \mN(\mz, \sigma^2(1) \mI)$ at $t=1$ and iteratively updated using following rule:
\begin{equation}\label{eqn:EM_cond}
\buh(t-\Delta t) = \buh(t) + (\sigma^2(t-\Delta t) - \sigma^2(t)) S_{\xi}(\buh(t), \buth, t) + \sqrt{\sigma^2(t) - \sigma^2(t-\Delta t)} \boldsymbol{z}, \quad \boldsymbol{z} \sim \mN(\mz, \mI).
\end{equation}
until $t=0$, and we let $\buh:=\buh(0)$. As discussed in \cite{saharia2022image}, the super-resolution task with a large magnification factor can be split into a sequence of SR tasks with smaller magnification factors. This approach enables parallel training of simpler models, each requiring fewer parameters and less training effort.

\subsection{Diffusion-Based Correction and Super-resolution}\label{sec:dcsr}
In this part, we detail the two-stage pipeline of the proposed DCSR method. First, an enhanced SDEdit with IPD removes the bias from the LFLR data.  Second, a cascaded SR3 network progressively upsamples the corrected output back to the target high resolution. To better illustrate the pipeline, we consider a setting where the HFHR dataset $\{ \buh_i \}_{i=1}^N$ consists of $N$ samples at a resolution of $256 \times 256$, while the LFLR dataset $\{ \bul_j \}_{j=1}^M$ consists of $M$ samples at a resolution of $32 \times 32$. This setup requires bias correction at the low-resolution  and an upsampling by a magnification factor of $8$. Following the strategy used in cascaded SR3~\cite{saharia2022image}, we decompose the overall SR task into three smaller SR tasks, each with a magnification factor of $2$.

Specifically, we define three restriction operators, $\mathcal{R}_1$, $\mathcal{R}_2$, and $\mathcal{R}$, which downsample the HFHR dataset $\{ \buh_i \}_{i=1}^N$ to resolutions of $128 \times 128$, $64 \times 64$, and $32 \times 32$, respectively, resulting in three paired datasets $\{ \buth_{128,i}, \buh_i \}_{i=1}^N$, $\{ \buth_{64,i}, \buth_{128,i} \}_{i=1}^N$ and $\{ \buth_{i}, \buth_{64,i} \}_{i=1}^N$. We then train three separate conditional diffusion models, denoted by $S_{\xi_1}$, $S_{\xi_2}$, and $S_{\xi_3}$, each using one of these datasets in parallel. These models are subsequently chained together to perform cascaded super-resolution, progressively refining the LR data from $32 \times 32$ back to the original $256 \times 256$ resolution. The training procedure for each SR model follows the steps outlined in Algorithm~\ref{alg:uh}. In addition, we train an unconditional diffusion model on the HFLR dataset $\{ \buth_i \}_{i=1}^N$, which is used to perform bias correction at the LR level. At inference step, given a LFLR dataset $\{ \bul_j \}_{j=1}^M$, select a metric $\mathcal{M}$, compute the optimal $t_1^*, t_2^*$ using Algorithm~\ref{alg:t1t2} with two datasets $\{ \bul_j \}_{j=1}^M$ and $\{ \buth_{i} \}_{i=1}^N$. The choice of $\mathcal{M}$ is critical and will be discussed in Section~\ref{sec:1d_adv}, as it is determined empirically. The LFLR data is then corrected using Algorithm~\ref{alg:dc}, producing a corrected dataset $\{ \hat{\mathbf{u}}^{l,*}_j \}_{j=1}^M$. Finally, the corrected dataset is refined using the cascaded SR3 ($S_{\xi_1}$, $S_{\xi_2}$ and $S_{\xi_3}$), resulting the final enhanced dataset, denoted as $ \{ \hat{\mathbf{u}}^h_j \}_{j=1}^M$. The entire procedure of DCSR is summarized in Algorithm~\ref{alg:dcsr}.

\begin{algorithm}[H]
\caption{Diffusion-based correction and super-resolution}
\label{alg:dcsr}
\begin{algorithmic}[1]
\REQUIRE Any LFLR dataset $\{ \bul_j \}_{j=1}^M$, a sufficiently large HFHR dataset $ \{ \buh_i \}_{i=1}^N$, noise scheduling function $\sigma(t)$, metric $\mathcal{M}$, three user-specific restriction operators $\mathcal{R}_1$, $\mathcal{R}_2$ and $\mathcal{R}$, batch size $B$ and max iteration $Iter$. 
\STATE {\textbf{Data preparation stage}:}
\STATE Generate three lower resolution datasets from HFHR dataset $ \{ \buh_i \}_{i=1}^N$ using $\mathcal{R}_1$, $\mathcal{R}_2$ and $\mathcal{R}$ respectively, resulting in $\{ \buth_{128, i} \}_{i=1}^N$, $\{ \buth_{64, i} \}_{i=1}^N$ and $\{ \buth_{i} \}_{i=1}^N$
\STATE {\textbf{Training stage}:}
\STATE Train SR models $S_{\xi_1}$, $S_{\xi_2}$ and $S_{\xi_3}$ using dataset $\{ \tilde{\mathbf{u}}_{128,i}^h, \buh_i \}_{i=1}^N$, $\{ \tilde{\mathbf{u}}_{64,i}^h, \buth_{128,i} \}_{i=1}^N$ and $\{ \tilde{\mathbf{u}}_{32,i}^h, \buth_{64,i} \}_{i=1}^N$ by Algorithm~\ref{alg:uh}
\STATE Train a diffusion model $S_{\theta}$ on empirical distribution $\{ \buth_{i} \}_{i=1}^N$ by running Algorithm~\ref{alg:utilde}.
\STATE {\textbf{Inference stage}:}
\STATE Obtain the corrected dataset $ \{ \hat{\mathbf{u}}^{l,*}_j \}_{j=1}^M$ at LR by Algorithm~\ref{alg:dc} 
\STATE Generate HR corrected dataset $ \{ \hat{\mathbf{u}}^h_j \}_{j=1}^M$ using cascaded conditional diffusion model $S_{\xi_1}$, $S_{\xi_2}$ and $S_{\xi_3}$ with EM method \eqref{eqn:EM_cond} iteratively.
\RETURN HR corrected dataset $ \{ \hat{\mathbf{u}}^h_j \}_{j=1}^M$
\end{algorithmic}
\end{algorithm}

\section{Numerical Experiments}\label{sec:numerical}
In this section, we present numerical examples to demonstrate the effectiveness of the proposed DCSR model in enhancing the fidelity of various LFLR datasets. We begin with a 1D advection equation to demonstrate the capability of IPD in eliminating various types of numerical errors introduced by different solvers, as well as noise-induced artifacts. In parallel, we also empirically investigate the choice of metric $\mathcal{M}$ in Algorithm~\ref{alg:dcsr} for determining the optimal values of $t^*_1$ and $t^*_2$. Next, we evaluate DCSR for enhancing the fidelity of data arising from fluid dynamics and elastic mechanics. In these cases, HFHR data are generated using reference numerical solvers on fine grid and LFLR data are generated using various numerical solvers on coarse grids. Details of the dataset generation procedures are provided in the respective subsections. Finally, we test DCSR on a real climate dataset, where the LFLR data is obtained by corrupting the HFLR data by various sources of noise.

\textbf{Baseline Methods.} The baseline methods used in the ablation study for the 1D example are described in Subsection~\ref{sec:1d_adv}. Here, we focus on a different set of baselines relevant to the PDE and climate settings. As discussed earlier, supervised learning frameworks such as neural operators are not applicable in this context due to the absence of paired datasets or known governing equations. Although methods based on OT and DDIB address the challenge of unpaired data, they are typically task-specific and do not generalize to unseen correction tasks. In this section, we consider the following baselines for comparison:  (1) $\text{LFLR+Interp}$, which upsamples the LFLR data using bicubic interpolation; (2) $\text{LFLR+SR}$, which upsamples the LFLR data using a pretrained cascaded SR3 model; (3) DCSR, our proposed method; and (4) the reference HFHR data.



\textbf{Metrics}. We introduce the metrics used in this paper to quantify the distance between two empirical distributions. These metrics will be used for selecting the optimal $t_1$ and $t_2$ in Algorithm~\ref{alg:t1t2} and for evaluating the quality of DCSR predictions compared to the reference HFHR data. We begin with the energy spectrum, which measures the energy of a function $f(\bx)$ in each Fourier mode and is defined as 
\[
E_f(k)  = \sum_{| \boldsymbol{k} | = k} \left| \sum_{n} f(\bx_n) \exp \left( -\mathbf{j} 2 \pi \boldsymbol{k} \cdot \bx_n / L \right) \right|^2
\]
where the $\bx_n$ denotes the grid points, the bold notation $\mathbf{j}$ denotes the imaginary number (not to be confused with the index subscript), and $k$ is the magnitude of the wave number $\boldsymbol{k}$. To quantify the overall alignment of two empirical distributions $\{ \bu_i \}_{i=1}^N$ and $\{ \bv_j \}_{j=1}^M$, we denote the average energy spectrum of two distributions by $E_\bu(k) = \frac{1}{N} \sum_{i} E_{\bu_i}(k)$ and $E_\bv(k) = \frac{1}{M} \sum_{j} E_{\bv_j}(k)$, and consider the mean energy log ratio (MELR):
\[
\text{MELR}(\bu, \bv)= \sum_k w_k \left| \log \left( \frac{E_\bu(k)}{E_\bv(k)} \right) \right|
\] 
Where $w_k$ is the user-specific weight function. If $w_k=1$ for all $k$, we refer it as the unweighted mean energy log ratio (MELRu). If $w_k=\frac{E_\bv(k)}{\sum_{j} E_\bv(j)}$, we refer it as the weighted mean energy log ratio (MELRw). Additional metrics used in this paper, including the Maximum Mean Discrepancy (MMD), Relative Mean Square Error (RMSE), Wasserstein-2 distance ($\mWa$) and Total Variation Distance (TVD), are provided in the Section~\ref{sec:metrics} in Appendix.

\textbf{Experiments Setup}. Denote the HFHR training dataset by $\{ \buh_i \}_{i=1}^N$, where $N$ is the size of the dataset. Bicubic interpolation is used for all restriction operators $\mathcal{R}_1$, $\mathcal{R}_2$ and $\mathcal{R}$. As stated in \ref{sec:dcsr}, four datasets will be generated after applying the restriction operators onto the HFHR dataset. The three paired datasets $\{ \buth_{128,i}, \buh_i \}_{i=1}^N$, $\{ \buth_{64,i}, \buth_{128,i} \}_{i=1}^N$ and $\{ \buth_{i}, \buth_{64,i} \}_{i=1}^N$ are used to train the cascaded conditional diffusion models $S_{\xi_1}$, $S_{\xi_2}$ and $S_{\xi_3}$ respectively. The HFLR dataset $\{ \buth_i \}_{i=1}^N$ is used to train the unconditional diffusion model $\mathcal{S}_{\theta}$ for performing bias correction at the LR level. We adopt the UNet architecture as described in \cite{song2020score} for all these four diffusion models and the details on the selection of the noise scheduling function and the hyperparameters can be found in the Appendix. In Algorithm~\ref{alg:dc}, the ODE solver (ODEsolve in line 4) employs the RK45 scheme with stopping criteria of $rtol=1e-5$ and $atol=1e-5$. In Algorithm~\ref{alg:dcsr}, the EM method in SR step (line 2) use total step $N_T=1000$. The selection of hyperparameters and the metric $\mathcal{M}$ in Algorithm~\ref{alg:t1t2} is non-trivial and is guided by the experimental results presented in Subsection~\ref{sec:1d_adv}. The selection of other hyperparamters is presented in Appendix.

\subsection{1D Advection}\label{sec:1d_adv}
In this example, we demonstrate the effectiveness of our enhanced SDEdit method with IPD in mitigating both artificial numerical errors introduced by numerical solvers and various noise corruptions. Consider the following toy problem:
\begin{equation}\label{eqn:adv}
\begin{cases} 
\partial_t u + v \partial_x u = 0, ~ ~x \in [0,1], \\ 
u(t, 0) = u(t, 1), \\ 
u(0, x) = u_0(x).
\end{cases}
\end{equation}
where and velocity is set to be $v=0.1$ and the initial conditions are defined as follow:
\[
u_0 (x) = \sum_{i=1}^K \chi_{[a_i, b_i]}(x),
\]
where:
\[
\chi_{[a_i, b_i]}(x) =
\begin{cases}
1, & \text{if } x \in [a_i, b_i] \subset [0, 1], \\
0, & \text{otherwise}.
\end{cases}
\]
and $K$ is a random variable that is uniformly distributed over the integers set $\{1, 2, 3\}$. We randomly sample $N=2,000$ initial conditions for training and $M=100$ for testing. To demonstrate the effectiveness of the DCSR model, we ensure that both the LF and HF solutions share the same spatial resolution, using a uniform mesh grid with $\Delta x = 1/100$. The training and test datasets consist of solutions of \eqref{eqn:adv} at $T=0.25$. The analytical solution to \eqref{eqn:adv}, which is accessible, serves as the HF data. To generate LF solutions, we set the time step size $\Delta t = 0.001$ and evolve the initial conditions to $T=0.25$ using three numerical schemes: the Godunov scheme (God), the Lax-Wendroff scheme (LW), and the Fourier spectral method (FFT). Additionally, we generate three more LF solutions by perturbing the HF solutions with three types of noise: White, Pink, and Brown noise, defined as:
\begin{equation}\label{eqn:noise}
  C_r \mathcal{F}^{-1} (\frac{\mathcal{F}(\beps)}{|\boldsymbol{k}|^{r/2}}),~\beps \sim \mN(\mz, \mI)  
\end{equation}
where $\mathcal{F}$ and $\mathcal{F}^{-1}$ are Fourier transform and inverse Fourier transform, and $\boldsymbol{k}$ is the wave number. Scalar $C_r$ denotes the magnitude of the noise and parameter $r$ determines the noise type: $r=0$ for White noise, $r=1$ for Pink noise and $r=2$ for Brown noise. Typically in this part, we let $C_r=0.1$.

We begin with an ablation study of three factors: 1. Comparison between IPD and BPD; 2. The effect of the metric $\mathcal{M}$ used to evaluate the distance between two empirical distributions in Algorithm~\ref{alg:dc}, specifically comparing MMD, MELRu, MELRw, and W2 ; 3. Impact of search ending time $T_e$ in Algorithm~\ref{alg:dc}. The study is conducted on all six types of LF data. The comparison of IPD and BPD performance across metrics and search ending times for the low-fidelity solution generated by the LW solver and corrupted with pink noise presented in Figure~\ref{fig:ablation}. Additional results for other cases are provided in the Appendix.
\begin{figure}[h]
    \centering
    \includegraphics[width=1\linewidth]{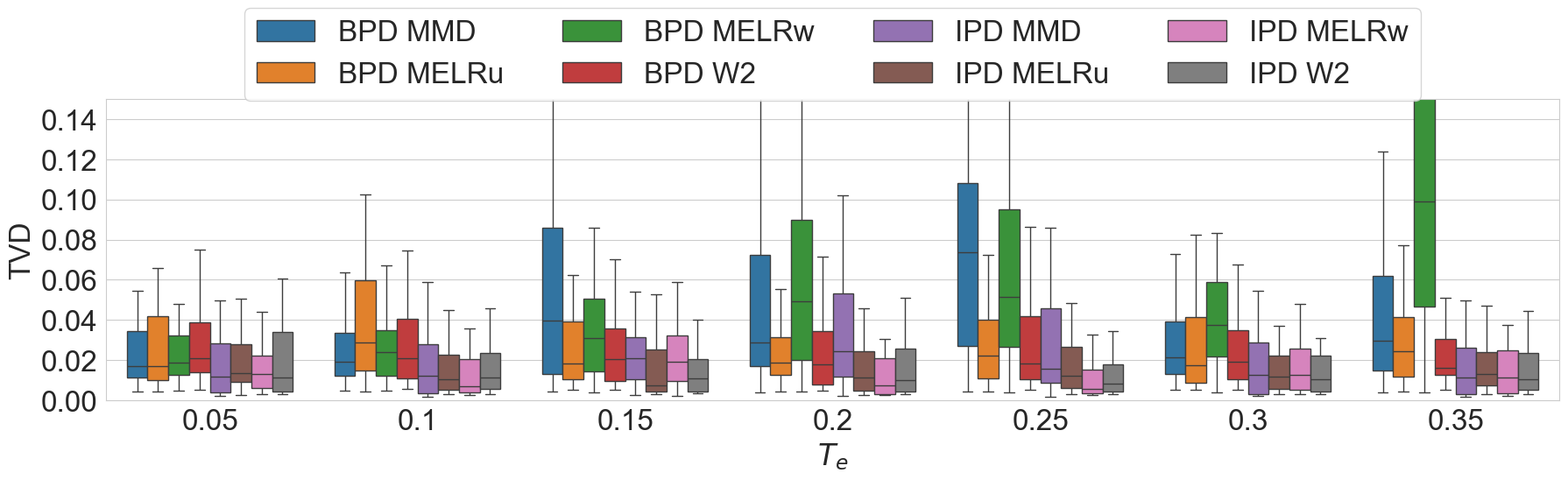}
    \includegraphics[width=1\linewidth]{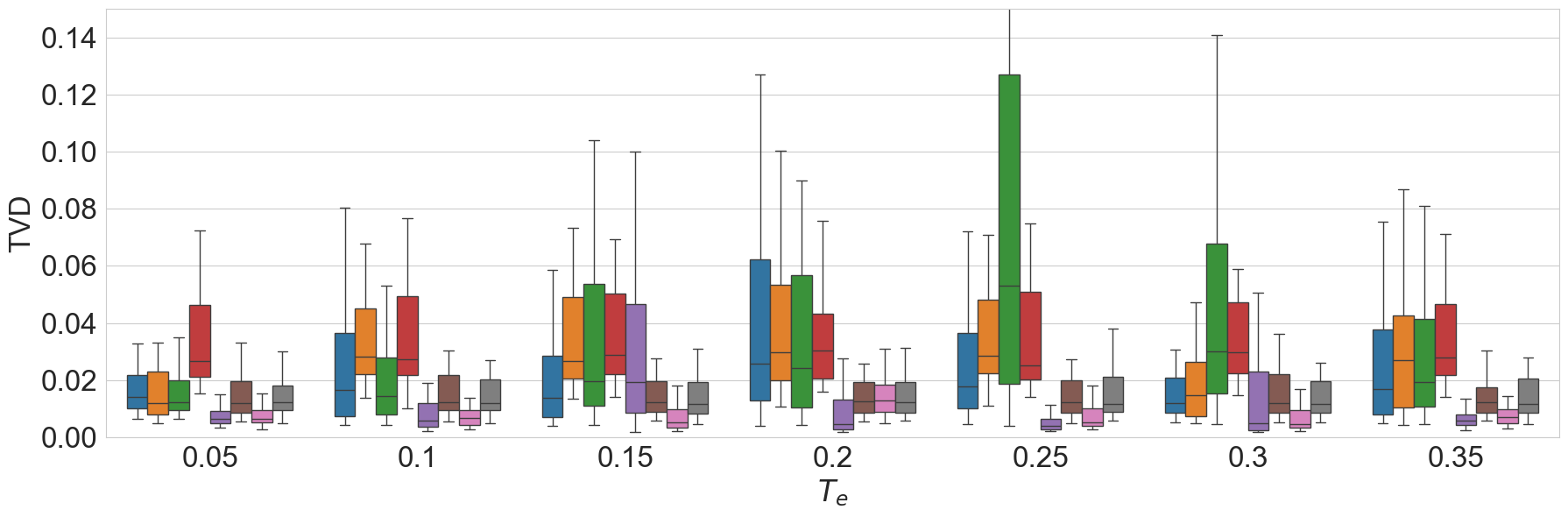}
    \caption{Box plot of comparison the TVD between correction obtained by Algotirhm~\ref{alg:dc} under different settings and reference. The settings include IPD and BPD, search ending times ($T_e$) and four metrics $\mathcal{M}$: MMD, MELRu, MELRw, and $\mWa$.}
    \label{fig:ablation}
\end{figure}
In Figure~\ref{fig:ablation}, the best overall settings can be identified by evaluating both the median TVD and the variability represented by the height of the boxes. Lower medians indicate better overall alignment between predictions and reference, and smaller box heights implies greater robustness. Notably, for the same metric and search ending time $T_e$, the IPD always outperform BPD, offering lower median TVD errors and more robust results (narrower boxes). Among the metrics $\mathcal{M}$,  MELRw using IPD stands out. Specifically, when the search ending time is set to $T_e=0.2$, it achieves the optimal balance between accuracy and robustness. Consequently, IPD with MELRw metric and $T_e=0.2$ is adopted in Algorithm~\ref{alg:dc} for the rest numerical examples. 

To further highlight the differences, Figure~\ref{fig:ipd_comp} directly compares the correction performance of BPD and IPD. The results show that BPD struggles to preserve the large-scale structure in some cases, indicating its lack of robustness. In contrast, the IPD corrections align more closely with the reference HF data. Several additional correction instances using the  setting above are presented in Figure~\ref{fig:1d_example}.

\begin{figure}
    \centering
    \includegraphics[width=0.99\linewidth]{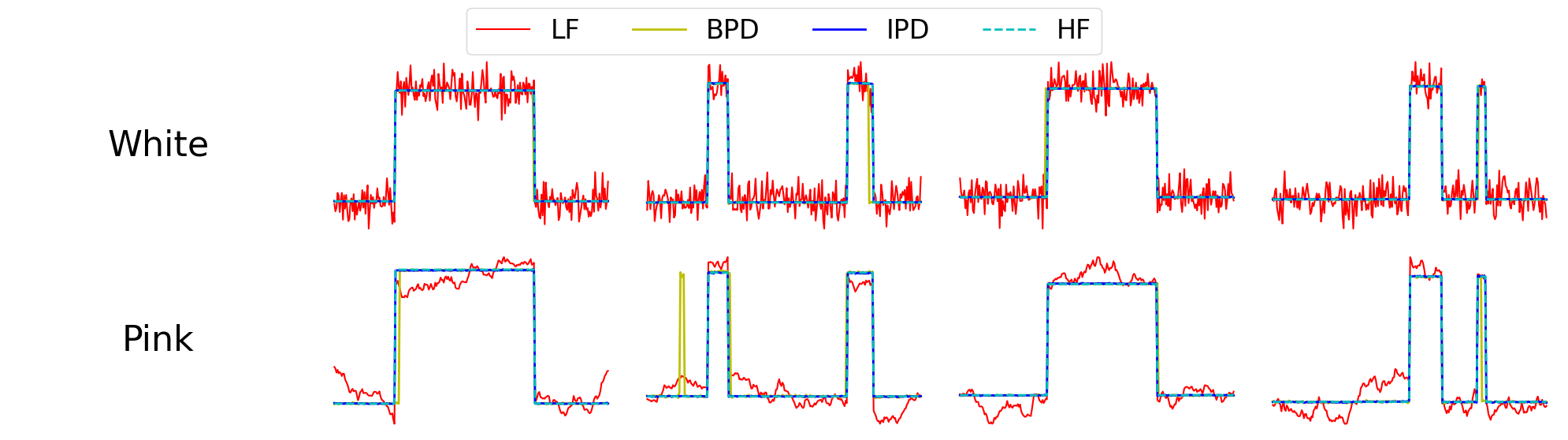}
    \caption{Comparison of correction obtained by Algorithm~\ref{alg:dc} using BPD (select $t^*$ by Algorithm~\ref{alg:t}) and IPD (select $t_1^*, t_2^*$ by Algorithm~\ref{alg:t1t2}) with MELRw metric and $T_e=0.2$ for LF data polluted with White and Pink noise. The red line represents the LF data, the yellow line corresponds to the correction using BPD, the blue line represents the correction using IPD, and the dashed green line shows the reference HF data.}
    \label{fig:ipd_comp}
\end{figure}

\begin{figure}[h]
    \centering
    \includegraphics[width=0.99\linewidth]{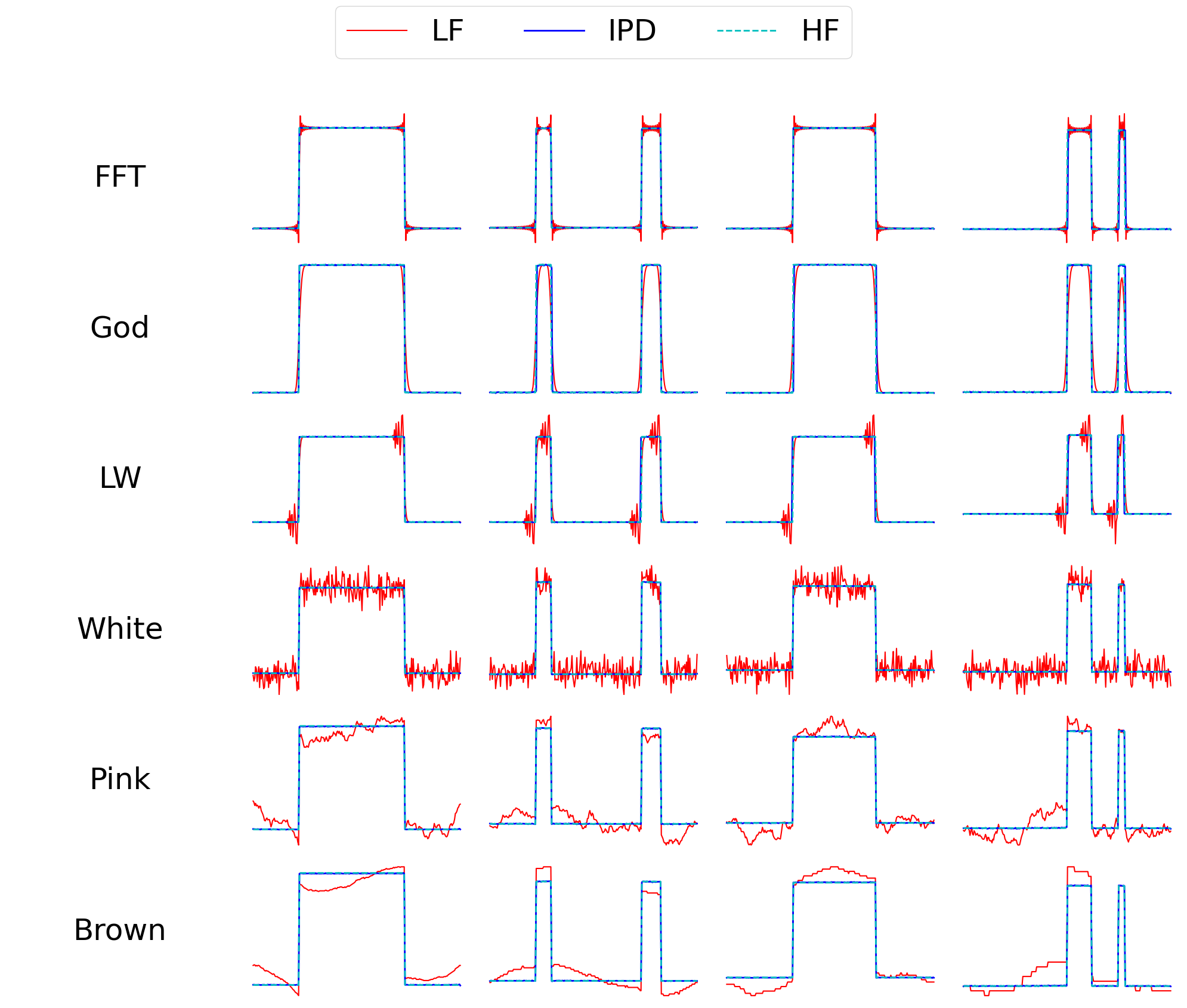}
    \caption{Corrections using IPD with the MELRw metric and $T_e=0.2$. The top three rows display corrections of LF solutions produced by three numerical solvers: the Fourier spectral method, Godunov scheme, and Lax-Wendroff scheme. The bottom three rows show corrections of LF data polluted by White, Pink, and Brown noise, respectively. The red lines represent the LF data, the blue lines show the correction using IPD, and the dashed cyan lines indicate the reference HF data.}
    \label{fig:1d_example}
\end{figure}

\newpage

\subsection{Linear Elasticity}
In this section, we discuss linear elasticity in computational micromechanics by analyzing deformation patterns in heterogeneous materials under external loadings. Specifically, we consider the strain field as the quantity of interest, which is obtained by solving the Lippmann-Schwinger equation \cite{LSequation} in a 2D fiber reinforced composite representative volume element (RVE) domain $\Omega \subset \mathbb{R}^2$ with spatial coordinates $\mathbf{x}=(x_1, x_2)$. The problem can be formulated as:
\begin{equation}
\begin{dcases*}
\varepsilon_{i j}(\mathbf{x})+\int_{\Omega} G_{i j k l}^{(0)}\left(\mathbf{x}, \mathbf{x}^{\prime}\right): \tau_{k l}\left(\mathbf{x}^{\prime}\right) \mathrm{d} \mathbf{x}^{\prime}-\bar{\varepsilon}_{i j}=0, \\
\tau_{i j}\left(\mathbf{x}\right)=\sigma_{i j}(\mathbf{x})-C_{i j k l}^{0}:\varepsilon_{k l}(\mathbf{x}), \\
\sigma_{i j}(\mathbf{x})=C_{i j k l}(\mathbf{x}): \varepsilon_{k l}(\mathbf{x}). \\
\end{dcases*}
\label{setup:balanceequation}
\end{equation}
The above relations are written in component-wise form ($ 1\leq i, j, k, l \leq 2$) using Voigt notation, which represent equilibrium and constitutive equations. Here, $\sigma_{i j}, \varepsilon_{i j}$ and $C_{i j k l}$ denote stress, strain and fourth-order elastic tensor, respectively. 
The elastic tensor $C_{i j k l}(\mathbf{x})$ is expressed in terms of Lamé constants $\lambda(\mathbf{x})$ and $\mu(\mathbf{x})$, such as:
\begin{equation}
C_{i j k l}(\mathbf{x})=\lambda(\mathbf{x}) \delta_{i j} \delta_{k l}+\mu(\mathbf{x})\left(\delta_{i k} \delta_{j l}+\delta_{i l} \delta_{j k}\right),
\label{setup:Lametransform}
\end{equation}
where $\delta_{i j}$ is the notation of Kronecker delta function. The selection of $\lambda(\mathbf{x})$ and $\mu(\mathbf{x})$ depends on the material type at  $\mathbf{x}$, i.e., whether it corresponds to the matrix or the fiber. $\tau_{i j}$ is defined as polarization stress, which denotes the difference between the real stress and the stress in the homogeneous reference material under the same strain. $G_{i j k l}^{(0)}$ denotes Green's operator corresponding to the reference material with elastic tensor $C_{i j k l}^{0}$. The explicit formula of $G_{i j k l}^{(0)}$ in Fourier space is demonstrated in reference \cite{WangLiuMicrometer}. $\bar{\varepsilon}_{i j}$ denotes far field strain, which is equal to average strain inside the RVE domain $\Omega$.

The first Lamé constants $\lambda$ of fiber and matrix materials are set to 4.85 GPa and 13.330 GPa, respectively, while the second Lamé constants $\mu$ are set to 1.368 GPa and 7.500 GPa, respectively. The squared RVE domain size is set as 50$\mu \text{m}$ $\times$ 50$\mu \text{m}$, where the volume fraction and radius of fiber are set as 40\% and 3.5$\mu \text{m}$, respectively. The RVE microstructures are generated using \texttt{rvesimulator} \cite{yi2023rvesimulator} according to random fiber generation algorithm $\texttt{RAND{\_}uSTRU{\_}GEN}$ proposed by Melro et al \cite{Melrogenerate2008}. 

To solve Eq.\eqref{setup:balanceequation}, macrostrain-controlled boundary condition is applied to the RVE, such as $\overline{\boldsymbol{\varepsilon}} = [1,0,0]^T$, and the solution is subsequently obtained using the fast Fourier transform (FFT)-based homogenization method with a fixed-point iteration solver \cite{Moulinec1998FFT}. The RVE domain is discretized with \(256 \times 256\) and  \(32 \times 32\) uniform pixels for high and low fidelity cases, respectively. More specifically, the high fidelity results are obtained using discrete Green's operator (DGO) for $G_{i j k l}^{(0)}$, and corresponding convergence criterion for fixed-point iteration solver is set to \(tol = 10^{-6}\). The low fidelity results are obtained using continuous Green's operator (CGO) for $G_{i j k l}^{(0)}$, as such an operator brings spurious oscillations and ringing artifacts, and corresponding convergence criterion for fixed-point iteration solver is set to \(tol = 1.5 \times 10^{-2}\).


\begin{figure}
    \centering
    \includegraphics[width=0.55\linewidth]{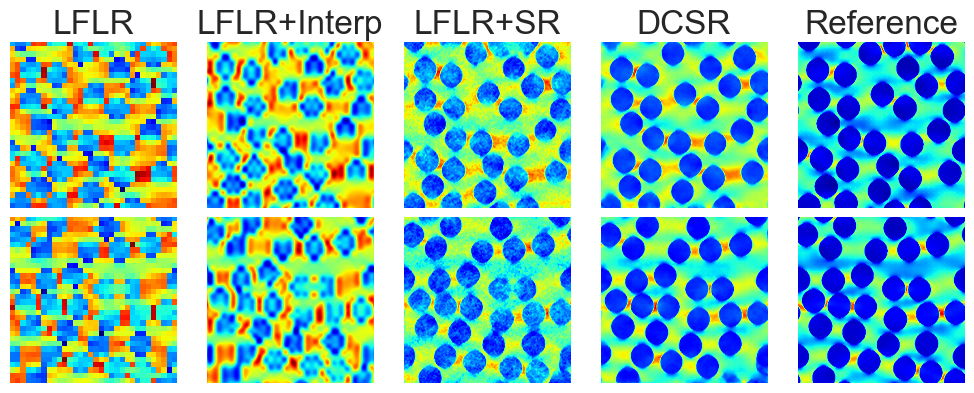}
    \includegraphics[width=0.38\linewidth]{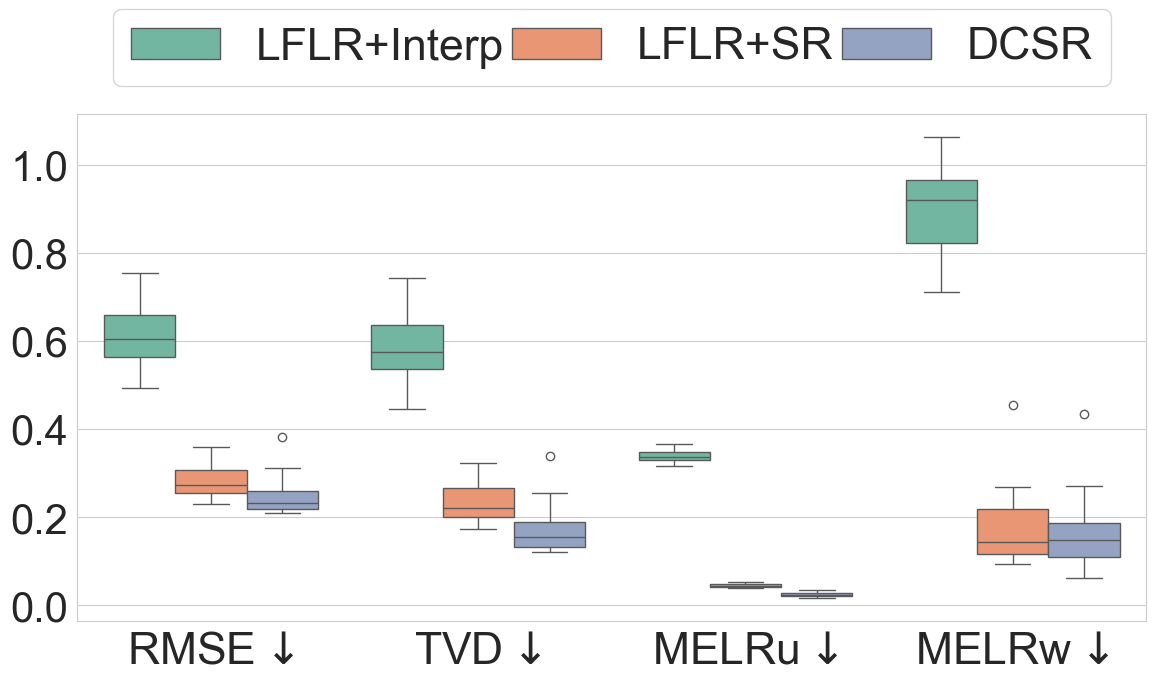}
    
    \includegraphics[width=0.55\linewidth]{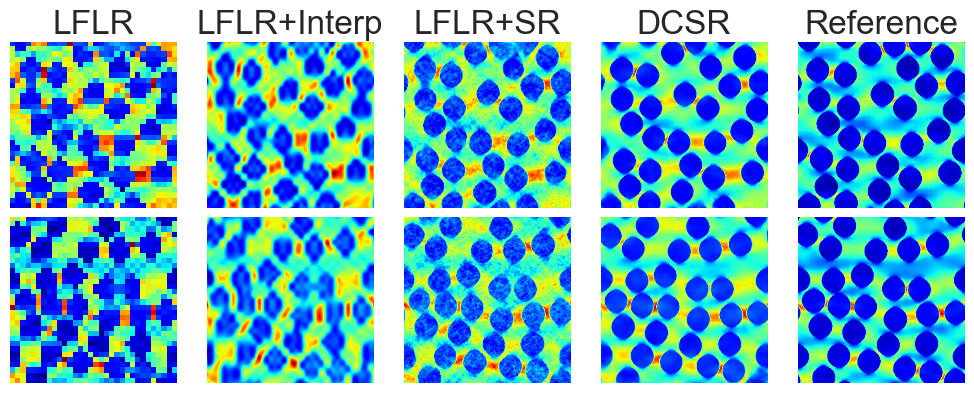}
    \hspace{0.01\linewidth}\includegraphics[width=0.37\linewidth]{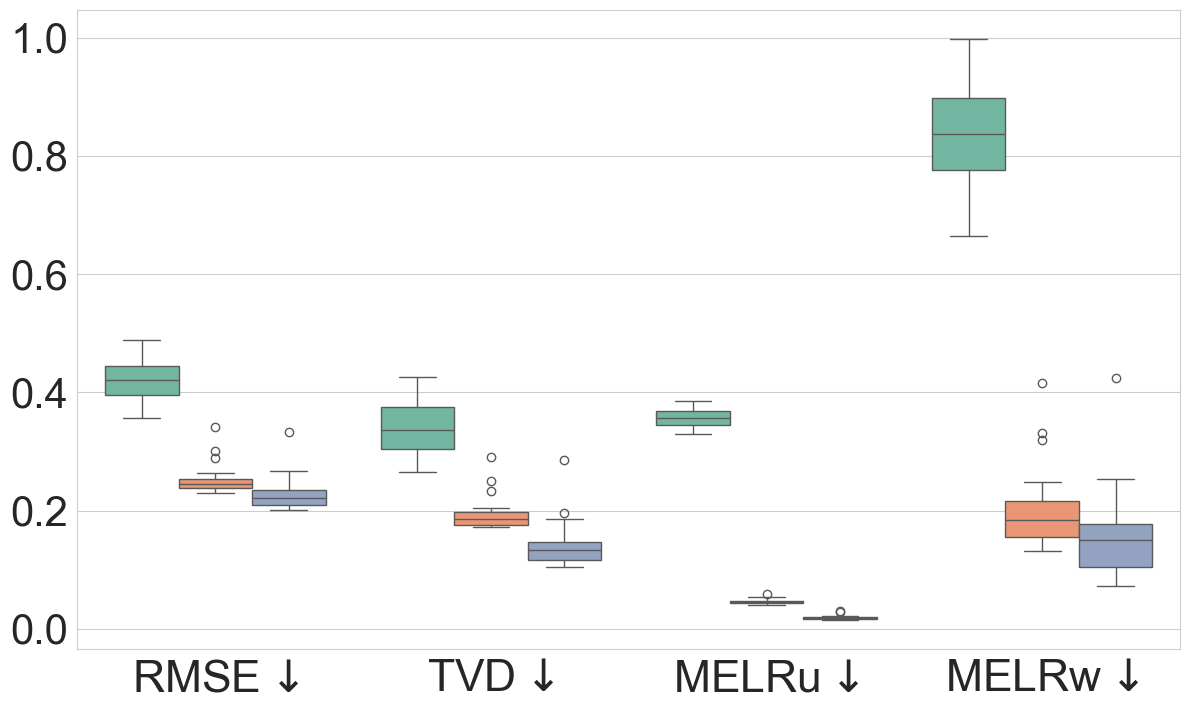}
    
    \caption{Strain field $\boldsymbol{\varepsilon}_{11}(\mathbf{x})$ distribution inside a RVE microstructure. The top row shows the results of low-fidelity data obtained using the continuous Green's operator, while the bottom row shows results of low-fidelity data obtained using the discrete Green's operator. The left columns compare low-fidelity data with results from various refinement methods. The right columns present boxplots comparing performance of various refinement methods across metrics.}
    \label{fig:ELAS_0}
\end{figure}

\begin{figure}
    \centering
    \includegraphics[width=0.55\linewidth]{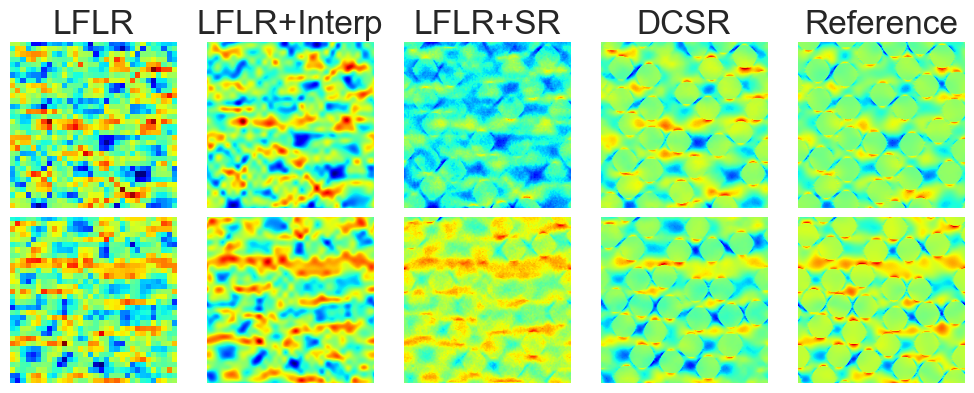}
    \includegraphics[width=0.38\linewidth]{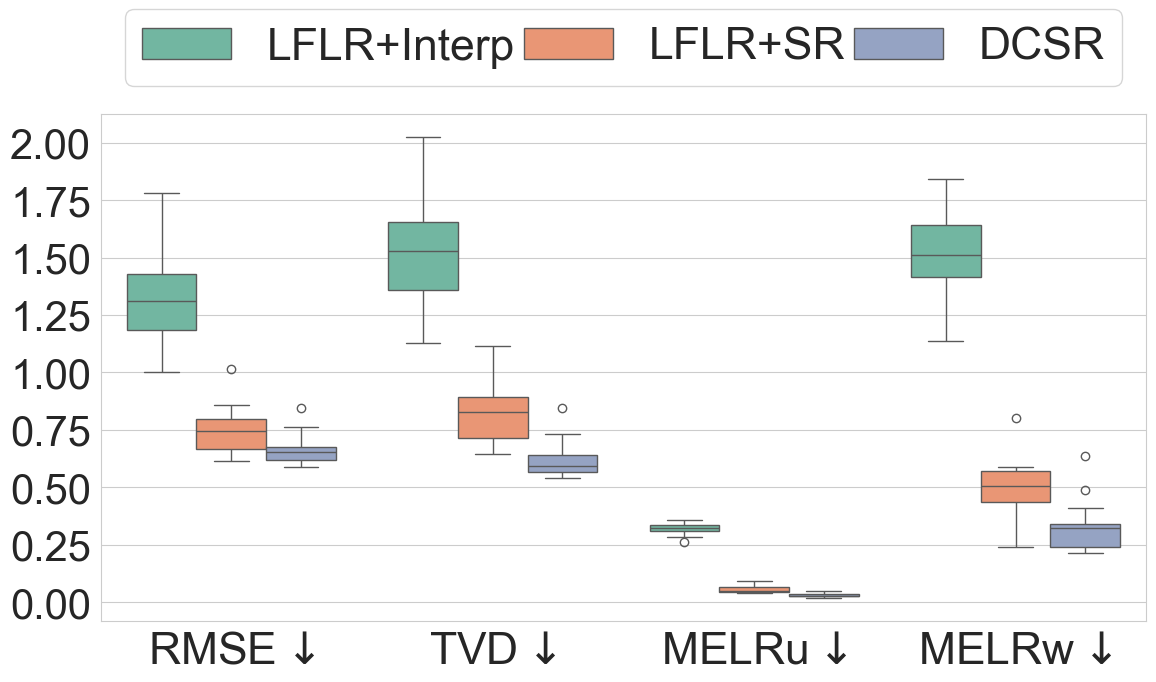}
    
    \includegraphics[width=0.55\linewidth]{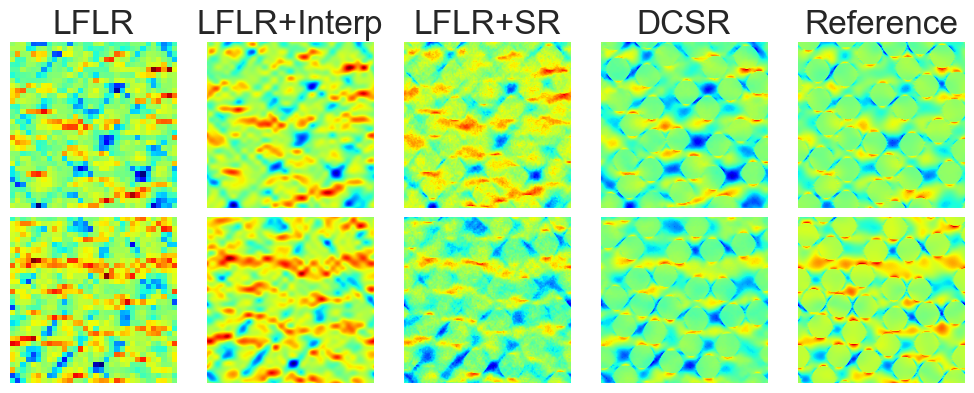}
    \hspace{0.01\linewidth}\includegraphics[width=0.37\linewidth]{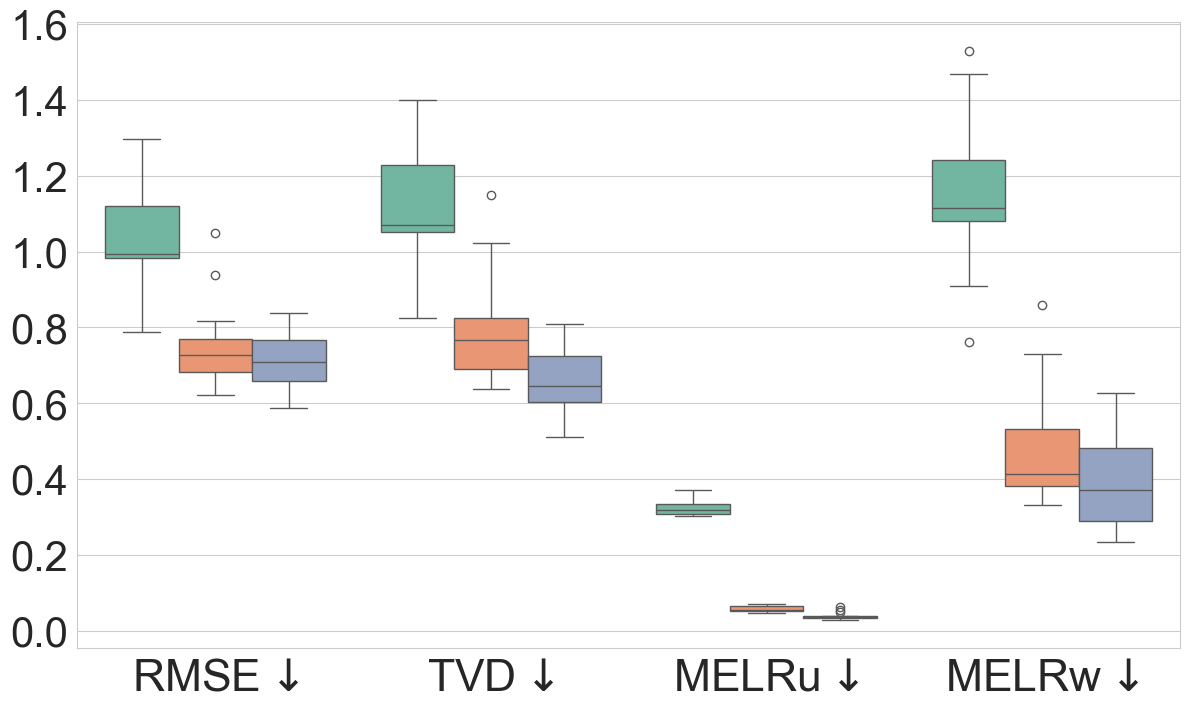}
    
    \caption{Strain field $\boldsymbol{\varepsilon}_{22}(\mathbf{x})$ distribution inside a RVE microstructure. The top row shows the results of low-fidelity data obtained using the continuous Green's operator, while the bottom row shows results of low-fidelity data obtained using the discrete Green's operator. The left columns compare low-fidelity data with results from various refinement methods. The right columns present boxplots comparing performance of various refinement methods across metrics.}
    \label{fig:ELAS_1}
\end{figure}

\newpage
\subsection{Navier-Stokes Equation}
In this example, we consider the vorticity form of the 2D Navier-Stokes equation with Kolmogorov force
\begin{equation}
\begin{cases} 
\partial_t \omega(t,\bx) + \mathbf{u}(t,\bx) \cdot \nabla \omega(t,\bx) - \nu \nabla^2 \omega(t,\bx) = f(\bx), \quad \bx \in [0, 2\pi]^2, t \in [0,T]\\
\nabla \cdot \mathbf{u}(t,\bx) = 0 , \quad \bx \in [0, 2\pi]^2, t \in [0,T] \\
w(0, \bx) = w_0(\bx), , \quad \bx \in [0, 2\pi]^2,
\end{cases}
\end{equation}
where $\bx = (x_1, x_2)$, $\omega(t,\bx)$ is the vorticity, $\mathbf{u}(t,\bx)=\nabla \times w(t, \bx)$ is the velocity field, and the Kolmogorov force is given by $f(\bx) =  \sin(k_0 x_1)$. In this example, we set the wavenumber \(k_0 = 2\) and viscosity \(\nu = 10^{-3}\). Functions with a resolution of \(256 \times 256\) are initially sampled from a log-normal distribution and evolved using a high-resolution solver for a time \(t_0 = 2\). The resulting state serves as the initial condition for later data generation.

From these initial conditions, the system is further evolved for a time \(T\) using the high-resolution solver to produce high-fidelity solutions. Low-fidelity solutions are generated by downsampling the high-fidelity data to a resolution of \(32 \times 32\) and evolving them using a low-resolution solver.

Both the low and high resolution solvers using pseudo-spectral methods implemented in \texttt{jax-cfd} \cite{Dresdner2022-Spectral-ML, kochkov2021machine}. The high-resolution solver operates on a \(256 \times 256\) uniform grid with a circular filter for stability, while the low-resolution solver, on a \(32 \times 32\) grid, using circular filter(Circ) and brick wall filter(BW). The results of these two example are presented in Figure~\ref{fig:ns_1}.

\begin{figure}[h]
    \centering
    \includegraphics[width=0.45\linewidth]{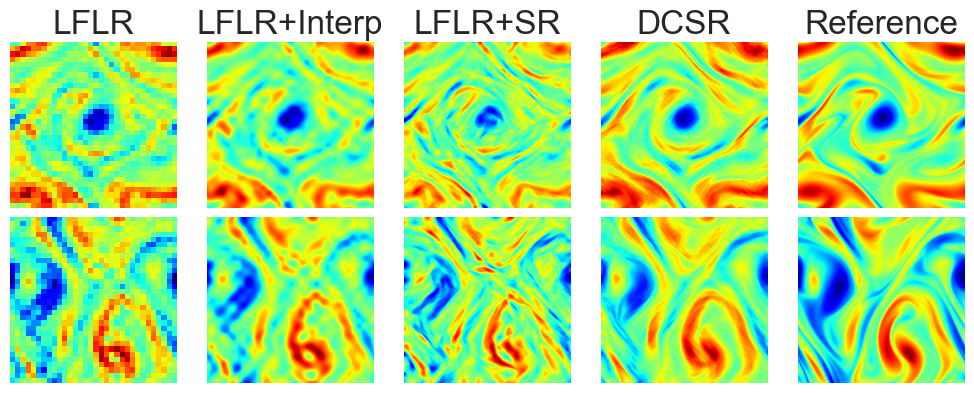}
    \includegraphics[width=0.27\linewidth, height=0.19\linewidth]{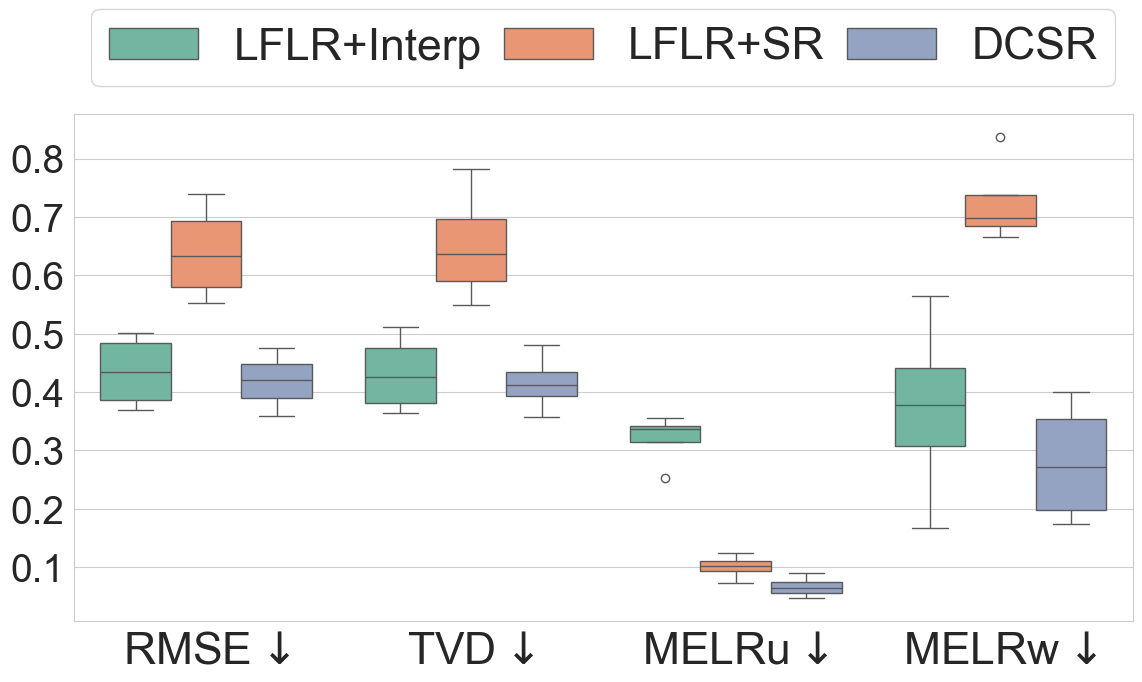}
    \includegraphics[width=0.27\linewidth, height=0.19\linewidth]{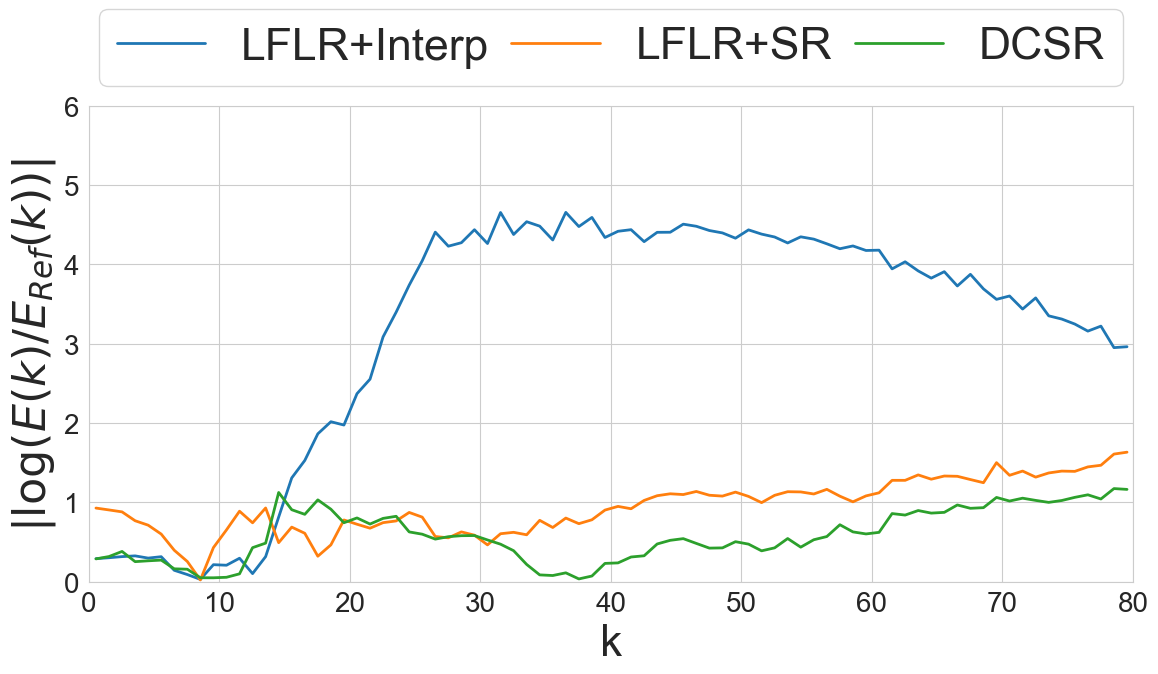}

    \includegraphics[width=0.45\linewidth]{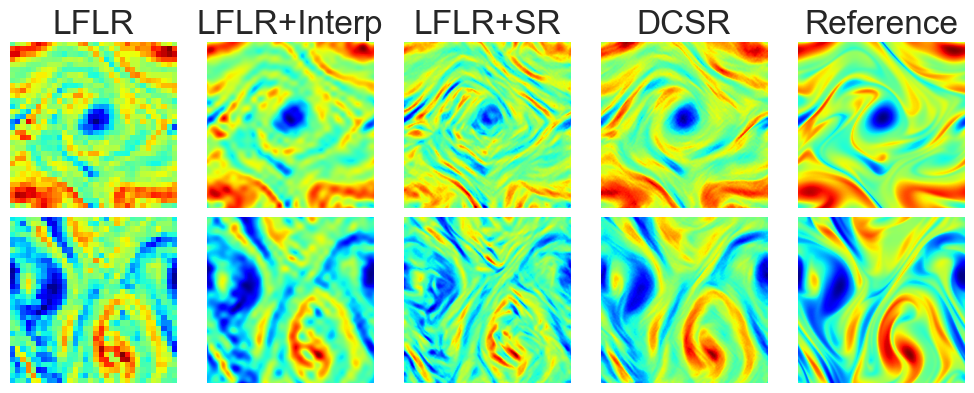}
    \includegraphics[width=0.27\linewidth, height=0.19\linewidth]{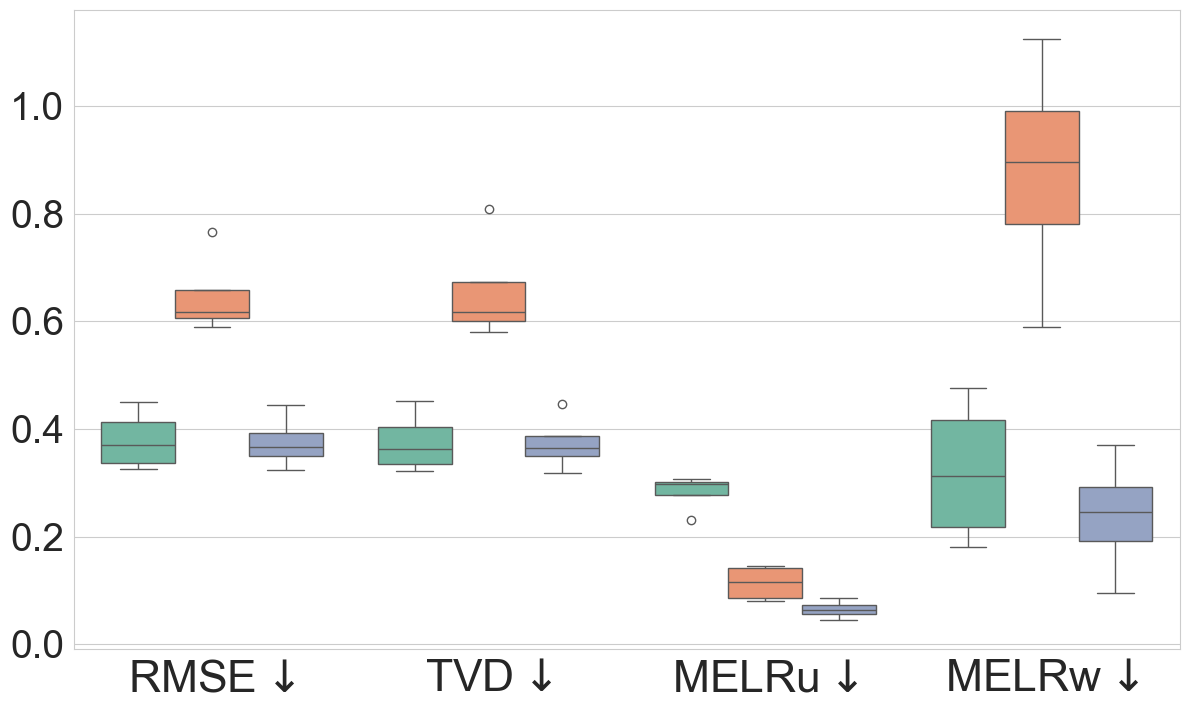}
    \includegraphics[width=0.27\linewidth, height=0.19\linewidth]{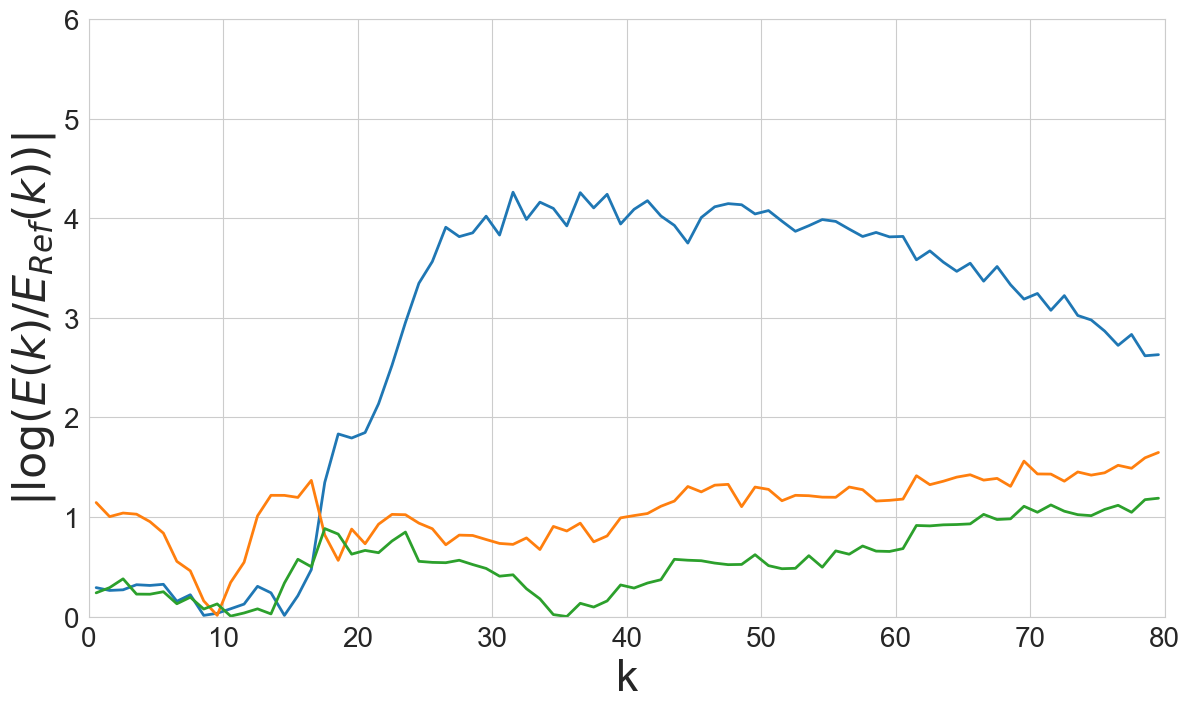}
    \caption{
    Numerical results for the Navier-Stokes equation. The top row shows the results of LFLR data obtained from a pseudo-spectral method with a circular filter, while the bottom row shows results of LFLR data obtained from a pseudo-spectral method with a brick wall filter. The left columns compare LFLR data with results from various refinement methods. The middle columns present boxplots comparing performance of various refinement methods across metrics. The right columns display the log ratio of the energy spectrum between the results produced by various refinement methods and the reference solution.}
    \label{fig:ns_1}
\end{figure}

\newpage
\subsection{Climate Data}\label{sec:climate}
In this subsection, we evaluate the effectiveness of DCSR on polluted climate data, focusing on Potential Vorticity (PV) and Wind Speed (WS) data from the ERA5 reanalysis dataset \cite{hersbach2020era5}. Specifically, we use hourly raw PV data at a pressure level of 400 hPa from the year 2023, and six-hourly WS magnitude at 10 meters above the surface from 2013 to 2023, both over the United States.

In our experiments, we randomly select $N=4000$ samples with a resolution of 256×256 as HFHR for training and $M=100$ for testing. To produce LFLR data from the $M=100$ testing samples, we first apply cubic interpolation to downsample the HFHR data to generate HFLR data with a resolution of 32×32. We then introduce three types of noise to generate polluted LFLR data. Specifically, we use three types of noise, White, Pink, and Brown, as defined in \eqref{eqn:noise}, with magnitude of $C_r=0.1$. We present the results of fidelity enhancement for data with white noise pollution in PV and WS using various fidelity enhancement methods in Figure~\ref{fig:CM_1}. The results for Pink and Brown noises are presented in Section~\ref{sec:supp} in Appendix.


\begin{figure}[h]
    \centering
    \includegraphics[width=0.55\linewidth]{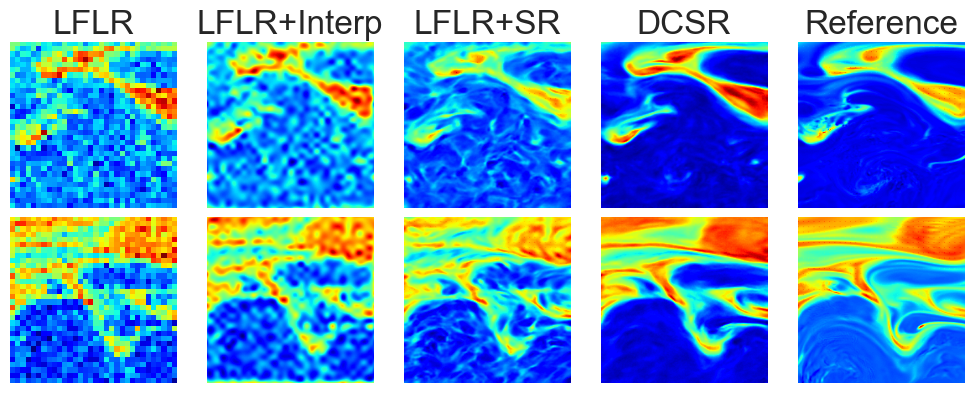}
    \includegraphics[width=0.38\linewidth]{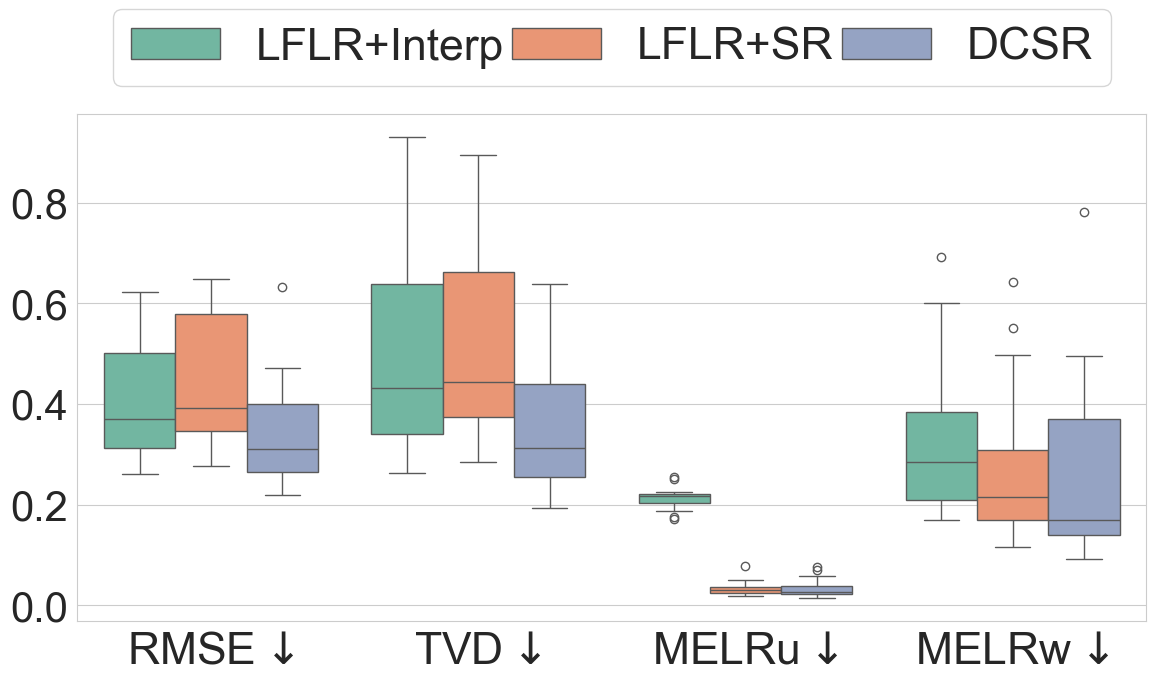}
    
    \includegraphics[width=0.55\linewidth]{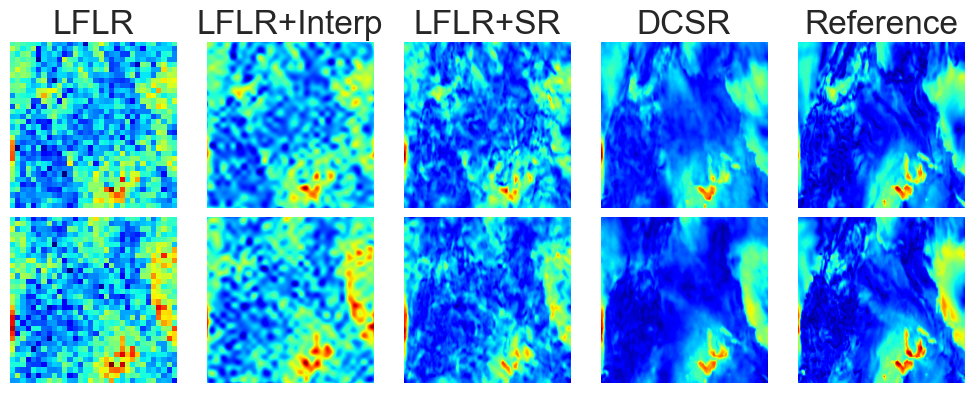}
    \includegraphics[width=0.38\linewidth]{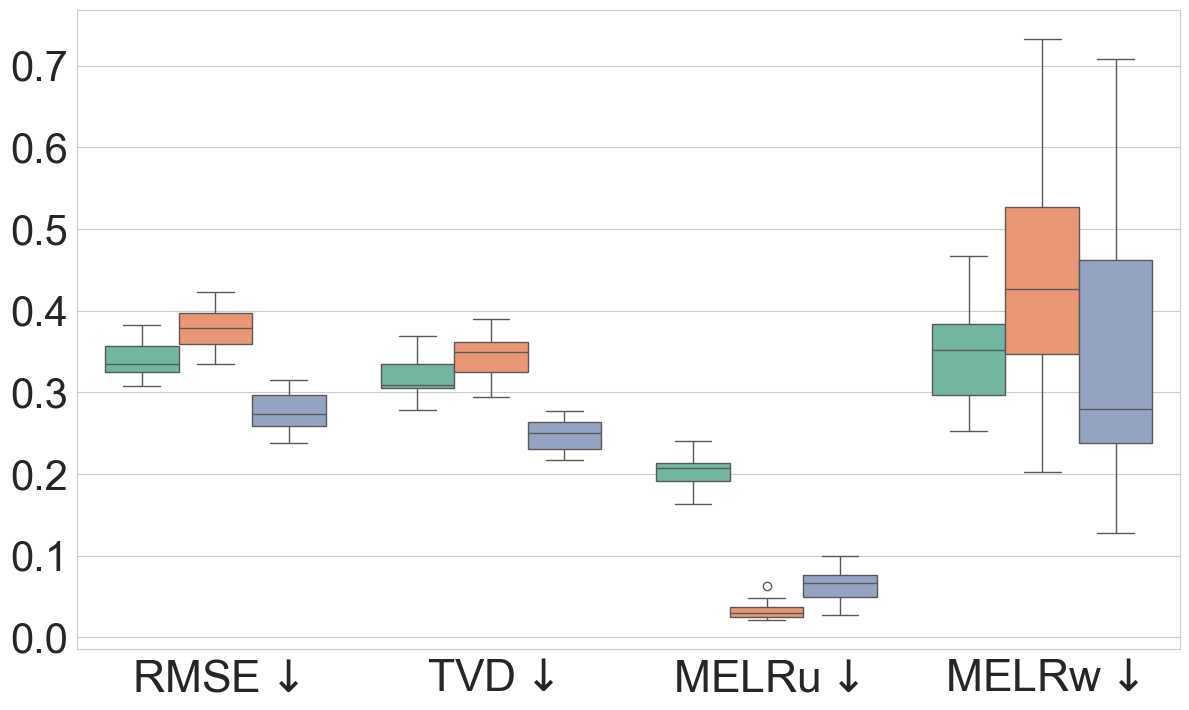}
    
    \caption{Top row presents results of low resolution potential vorticity data polluted with White noise and the bottom shows the results of LFLR wind speed polluted with White noise. The left columns compare LFLR data with results from various refinement methods. The right columns present boxplots comparing performance of various refinement methods across metrics.}
    \label{fig:CM_1}
\end{figure}

\section{Conclusion and Future Work}
We propose a purely data-driven method DCSR to enhance data fidelity, improving LFLR data to HFHR data. This method addresses three difficulties that most existing frameworks fails to overcome: 1. lack of underlying physics information; 2. The LFLR data cannot be treated as a simple downsampling of HFHR data, as it contains inherent biases that must be corrected; 3. the need for a unified approach capable of correcting biases due to diverse error sources. To address these challenges, we employ an enhanced SDEdit method incorporating a novel imbalanced perturbing and denoising strategy at the LR level to effectively correct various types of biases without relying on physical models. This step resolves the aforementioned challenges, with its effectiveness and robustness validated through theoretical analysis and experimental results. Finally, the corrected LR data is further refined using a cascaded SR3 model, recovering the desired small-scale details thus ensuring the alignment with the statistical properties of the HFHR data.

The versatile improvement capability of the proposed method makes it suitable for a wide range of data fidelity enhancement tasks. One notable application is in scientific computing, where the DCSR offers an efficient alternative to the computationally expensive HR solvers. LFLR data can first be generated using any LR solver and then enhanced using DCSR to improve fidelity. While the prediction of DCSR may still contain errors, they can serve as a good initial guess that can be further refined using HR solvers. A similar idea has recently been explored in  \cite{Lu2024}.

The second potential application is for climate prediction, where observational data is often contaminated by various types of noise. The DCSR offers a promising approach to correct these noises at low resolution and then downscale the corrected data to the HR counterparts. This method offers a cost-effective alternative to expensive real-world data collection, making it highly valuable for weather prediction. A potential future research direction is to chain the DCSR to iteratively enhance data fidelity. Specifically, the method could first improve low-fidelity data at low resolution (e.g., the NCEP Reanalysis \cite{kalnay2018ncep} with a spatial resolution of approximately 210 km), then refine it to medium resolution (e.g., the ERA5 Reanalysis \cite{hersbach2020era5} with a spatial resolution of approximately 31 km), and finally downscale it to high resolution (e.g., the Wind Integration National Dataset \cite{draxl2015overview} with a spatial resolution of approximately 2 km). We leave  the performance calibration of the proposed approach in these applications for future investigations.

\section*{Acknowledgment} 
YL thanks the support from the National Science Foundation through the award
DMS-2436333 and the support from the Data Science Initiative at University of
Minnesota through a MnDRIVE DSI Seed Grant.

\bibliographystyle{plain}
\bibliography{reference.bib}

\newpage 
\appendix
\section{Notation Table}
\begin{table}[h]
    \centering
    \begin{tabular}{c c}
    \hline
      Notation   &   Description \\
    \hline
    $\buh$    &   High fidelity high resolution (HFHR) data \\
     $\buth$    &  Low resolution version of high fidelity (HFLR) data\\
     $\bul$ & Low fidelity low resolution (LFLR) data\\
     $\be$ & Bias within LFLR data $\bul=\buth +\be$ \\
     $\mathcal{R}$ & Restriction operator that $\buth = \mathcal{R} (\buh)$ \\
     $\sigma(t)$ & Noise scheduling function\\
     $S_{\theta}(\buth(t), t)$ & Unconditional score model used to approximate $P(\buth)$  \\
     $S_{\xi_1}(\buth(t), \tilde{\mathbf{u}}_{128}^h, t)$ & Conditional score model used to approximate $P(\buth|\tilde{\mathbf{u}}_{128}^h)$  \\
     $S_{\xi_2}(\tilde{\mathbf{u}}_{128}^h(t), \tilde{\mathbf{u}}_{64}^h, t)$ & Conditional score model used to approximate $P(\tilde{\mathbf{u}}_{128}^h|\tilde{\mathbf{u}}_{64}^h)$  \\
     $S_{\xi_3}(\tilde{\mathbf{u}}_{64}^h(t), \buth, t)$ & Conditional score model used to approximate $P(\tilde{\mathbf{u}}_{64}^h|\buth)$  \\
     $L(\theta)$ & Loss function for $S_{\theta}$ defined in \eqref{eqn:loss_dc}\\
     $t_1$ & Forward perturbing time \\
     $t_2$ & Backward denoising time \\
     $\buhat(t_1, t_2)$ & The correction of $\bul$ using IPD with $t_1$ and $t_2$. \\
     $t^*_1,t^*_2$ & Optimal perturbing and denoising time selected by Algorithm~\ref{alg:t1t2} \\
     $\buhat(t^*_1, t^*_2)$ & The correction of $\bul$ using IPD with optimal $t^*_1$ and $t^*_2$. \\
     $\hat{\mathbf{u}}^h(t^*_1, t^*_2)$ & The enhancement of DCSR for $\bul$correction using IPD with optimal $t^*_1$ and $t^*_2$. \\
     $N$ & Number of training samples\\
     $M$ & Number of testing samples\\
     $\text{ODEsolve}(\bx(t_1), t_1, t_2; v)$ & The solution of ODE driven by velocity field $-\frac{1}{2} \frac{\rd [\sigma^2(t)]}{\rd t} v(\bx(t), t)$, \\ & starting from initial condition $\bx(t_1)$ at $t_1$ and evolving to $t_2$. \\
     BPD & Balanced perturbing and denoising process\\
     IPD & Imbalanced perturbing and denoising process\\
     TVD & Total Variation Distance \\
     RMSE & Relative root mean squared error \\
     MMD & Maximum Mean Discrepancy\\
     $\mWa$ & Wasserstein-2 distance\\
     MELRu & Unweighted mean energy log ratio \\
     MELRw & Weighted mean energy log ratio \\
     \hline
    \end{tabular}
    \caption{Table of Notations}
    \label{tab:notations}
\end{table}

\section{Hyperparameters of Neural Networks}
The noise scheduling function $\sigma(t)$ is selected as $\sigma(t) = \sqrt{\frac{\sigma^{2t} -1}{2 \log \sigma}}$ and the dimension of Gaussian random feature embeddings used in the time dependent UNet is 128. The remaining hyperparameters of diffusion models are presented in Table~\ref{tab:hyper}.
\begin{table}[h]
    \centering
    \begin{tabular}{|c|c|c|c|c|c|}
        \hline  
        & \makecell{$\mathcal{S}_{\theta}$(1D)} & \makecell{$\mathcal{S}_{\theta}$(2D)} & \makecell{$\mathcal{S}_{\xi_1}$} & \makecell{$\mathcal{S}_{\xi_2}$} & \makecell{$\mathcal{S}_{\xi_3}$}  \\
        \hline
        $\sigma$ & 25 & 25 & 50 & 50 &50  \\
        \hline
        \hline
        Base Channel & 64 & 64 &64 &128 &256 \\
        \hline
        \makecell{Down and Up \\ Channels multipliers} & 1, 2, 4 & 1, 2, 4, 8 & 1, 2, 4, 8 & 1, 2, 4, 8 & 1, 2, 4, 8 \\
        \hline
        Middle Channel & [256, 256] & [512, 512] & [512, 512] & [1024, 1024] & [2048, 2048] \\
        \hline
        \hline
        Batch size & 64 & 64 & 64 & 32 & 20 \\
        \hline
        Learning rate & 1e-3 & 5e-4 & 5e-4 & 5e-4 & 5e-4 \\
        \hline
    \end{tabular}
    \caption{Table of diffusion models hyperparameters}
    \label{tab:hyper}
\end{table}

\section{Proofs of Theoretical Results}\label{sec:proof}
Before presenting the proofs, we first introduce two notations. 
\begin{itemize}
    \item The \textit{correction} of LFLR data $\bul$ using IPD, which involves running the forward time SDE for $t_1$ and the backward PF ODE for $t_2$, is denoted by
    \begin{equation}
    \buhat(t_1, t_2) = \text{ODEsolve} (\bul(t_1), t_2, 0; S_{\theta}).
    \end{equation}

    \item The \textit{reconstruction} of HFLR data $\buth$ using IPD, which involves running the forward time SDE for $t_1$ and the backward PF ODE for $t_2$, is denoted by
\begin{equation}
\hat{\mathbf{u}}^h(t_1, t_2) = \text{ODEsolve} (\buth(t_1), t_2, 0; S_{\theta}).
\end{equation}
\end{itemize}

\subsection{Proof of Proposition~\ref{pro1}}
Before providing proof to Proposition~\ref{pro1}, we first establish the following two lemmas.
\begin{lemma}\label{lemma1}
    For two ODEs with different velocity function but with same initial conditions, i.e.
    \[
    \left\{
    \begin{array}{ll}
    \frac{\rd \bx_f(t)}{\rd t} = f(\bx_f, t), \quad t \in [0, 1]\\
    \bx_f(0) = \bx_0
    \end{array}
    \right.
    \]
    and 
    \[
    \left\{
    \begin{array}{ll}
    \frac{\rd \bx_g(t)}{\rd t} = g(\bx_g, t), \quad t \in [0, 1]\\
    \bx_g(0) = \bx_0
    \end{array}
    \right.
    \]
    If $g$ is $L$-Lipchitz in $\bx_g$, we have
    \[
    \| \bx_f(t) - \bx_g(t) \| \leq e^{Lt} \int_0^t  \|f(\bx_f, s) - g(\bx_f, s)\|_2 \rd s.
    \]
    \begin{proof}
        Let $h(t) = \bx_f(t) - \bx_g(t)$, then we have 
        \[
        \frac{d h(t)}{d t} = f(\bx_f, t) - g(\bx_g, t)
        \]
        which implies
        \[
        \|\frac{d h(t)}{d t}\|_2 \leq \|f(\bx_f, t) - g(\bx_f, t)\|_2 + \|g(\bx_f, t) - g(\bx_g, t)\|_2
        \]
        along with the fact that
        \[
        \frac{d \| h(t)\|_2}{d t} \leq \|\frac{d h(t)}{d t}\|_2
        \]
        and $g$ is $L$-Lipchitz continuous in $\bx_g$
        we get
        \[
        \frac{d \| h(t)\|_2}{d t} \leq \|f(\bx_f, t) - g(\bx_f, t)\|_2 + L \| h(t) \|_2
        \]
        multiply by $e^{-Lt}$ we have
        \[
        \frac{d [e^{-Lt} \| h(t)\|_2]}{d t} \leq e^{-Lt} \|f(\bx_f, t) - g(\bx_f, t)\|_2
        \]
        integrate from $0$ to $t$ we get
        \[
        e^{-Lt} \| h(t)\|_2 - 0 \leq \int_0^t e^{-Ls} \|f(\bx_f, s) - g(\bx_f, s)\|_2 \rd s
        \]
        which implies 
        \[
        \| h(t)\|_2 \leq e^{Lt} \int_0^t  \|f(\bx_f, s) - g(\bx_f, s)\|_2 \rd s
        \]
    \end{proof}
\end{lemma}

\begin{lemma}\label{lemma2}
    For a $\buth \sim p(\buth)$ and a $t \in (0, 1)$, if training loss defined in equation \eqref{eqn:loss} satisfies $L<\delta$, then the expectation of $L^2$ distance between $\buth$ and its reconstruction $\hat{\mathbf{u}}^h(t, t)$  is bounded by:
    \begin{equation}
       \mE_{\buth \sim p(\buth), \varepsilon \sim \mN(\mz, \mI)} [\| \buth - \hat{\mathbf{u}}^h(t, t) \|_2^2 ] \leq e^{2L_s t} \sigma^2(t) \delta.
    \end{equation}
\end{lemma}
\begin{proof}
    \begin{align*}
        \| \buth - \hat{\mathbf{u}}^h(t,t) \|_2^2 & = \| \text{ODEsolve} (\buth+\sigma(t)\beps, t, 0; \nabla_{\tilde{\mathbf{u}}} \log p(\tilde{\mathbf{u}}(t)))  \\
        & \quad \quad - \text{ODEsolve} (\buth+\sigma(t)\beps, t, 0; S_{\theta}) \|_2^2 
    \end{align*}
    By Lemma~\ref{lemma1}, Cauchy–Schwarz inequality and the fact that $\sigma(0)=0$, we have
    \begin{align*}
        \| \buth - \hat{\mathbf{u}}^h(t,t) \|_2^2 & \leq e^{2L_s t} \left( \int_t^0 \| \sigma'(s) \sigma(s) (  \nabla_{\tilde{\mathbf{u}}} \log p(\tilde{\mathbf{u}}(s)) -  S_{\theta} (\tilde{\mathbf{u}} + \sigma(s) \beps, s) ) \|_2 \rd s \right)^2 \\
        & = e^{2L_s t} \left( \int_t^0 \| \sigma'(s) \sigma(s) (  -\frac{\buth(s)-\buth}{\sigma(s)} -  S_{\theta} (\tilde{\mathbf{u}} + \sigma(s) \beps, s) ) \|_2 \rd s \right)^2 \\
        & \leq e^{2L_s t} \sigma^2(t) \int_0^1 \| \beps + \sigma(s) S_{\theta} (\tilde{\mathbf{u}} + \sigma(s) \beps, s)  \|_2^2 \rd s
    \end{align*}
    Apply the expectation we have
    \begin{align*}
    \mE_{\buth \sim p(\buth), \varepsilon \sim \mN(\mz, \mI)} [\| \buth - \hat{\mathbf{u}}^h(t,t) \|_2^2 ] & \leq e^{2L_s t} \sigma^2(t) \mE_{\buth \sim p(\buth), \varepsilon \sim \mN(\mz, \mI)} [ \int_0^1 \| \beps + \sigma(s) S_{\theta} (\tilde{u} + \sigma(s) \beps, s)  \|_2^2 \rd s ]\\
    & \leq e^{2L_s t} \sigma^2(t) \delta
    \end{align*}
\end{proof}

\addtocounter{proposition}{-2}

\noindent \textbf{The proof of Proposition~\ref{pro1}} 
\setcounter{proposition}{0}
\begin{proposition}
    For $\buth \sim p(\buth)$, and the low fidelity data $\bul = \buth + \be$ with the error term $\be \sim p(\be)=\mN(\mz, \gamma^2 \mI) $, if the empirical training loss defined in equation \eqref{eqn:loss} satisfies $L(\theta)<\delta$, then for $0<t_1<t_2<1$ such that $\sigma^2(t_2) =\sigma^2(t_1) + \gamma^2$, the expected $L^2$ distance between $\buth$ and its reconstruction $\buhat(t_1, t_2)$  is bounded by:
    \begin{equation}
        \mE_{\buth \sim p(\buth), \beps \sim \mN(\mz, \mI), \be \sim p(\be)}(\| \buhat(t_1, t_2) - \buth \|^2_2) \leq e^{2L_st_2} \sigma^2(t_2) \delta.
    \end{equation}
\end{proposition}
\begin{proof}
Since the injected perturbation $\sigma(t_1) \beps$ and the bias $e$ are independent, their sum can be represented as a single zero-mean Normal variable with the variance $\sigma^2(t_1) + \gamma^2$, which is equal to $\sigma^2(t_2)$. Thus we have
    \begin{align*}
        \mE_{\buth \sim p(\buth), \beps \sim \mN(\mz, \mI), \be \sim p(\be)} & (\| \buhat(t_1, t_2) - \buth \|^2_2) \\
        &=\mE_{\buth \sim p(\buth), \beps \sim \mN(\mz, \mI)} (\| \hat{\mathbf{u}}^h(t_2, t_2) - \buth \|^2_2) \\
        & \leq e^{2L_st_2} \sigma^2(t_2) \delta.
    \end{align*}
\end{proof}

\addtocounter{theorem}{-1}
\noindent \textbf{Proof of Theorem~\ref{thm}}

\begin{theorem}
    For a HFLR data $\buth \sim p(\buth)$, let $\bul = \buth + e$ be a LFLR data, where $e$ is the bias. Assume $0<t_1<t_2<1$. Then for all $\lambda \in (0,1)$, with probability at least $1-\lambda$, the $L^2$ distance between the target HFLR data $\buth$ and the correction $\buhat(t_1, t_2)$ is bounded by
    \begin{align*}
        \| \buhat(t_1, t_2) - \buth \|^2_2 \leq e^{2L_s t_2} [\| e \|_2^2 + \sigma^2(t_2)\delta + (\sigma^2(t_1) + \sigma^2(t_2)) C_{\lambda}].
    \end{align*}
    where $C_{\lambda}=d + 2 \sqrt{d \log\frac{1}{\lambda}} + 2\log\frac{1}{\lambda}$.
\end{theorem}
\begin{proof}
\begin{align}
    \| \buhat(t_1, t_2) - \buth \|_2^2 &
    \leq \|\buhat(t_1, t_2) - \hat{u}^h(t_2, t_2) \|_2^2 + \|\hat{u}^h(t_2, t_2)  - \buth \|_2^2 \label{eqn:total_error}
\end{align}
To bound the first term in \eqref{eqn:total_error}, we leverage the stability of the backward PF ODE \eqref{eqn:pfode_notation}, which gives:
\begin{align}
    \|\buhat(t_1, t_2) - \hat{\mathbf{u}}^h(t_2, t_2) \|_2^2  &= \| \text{ODEsolve}(\buth+\be+\sigma(t_1) \beps_1, t_2, 0; S_{\theta}) - \text{ODEsolve} (\buth+\sigma(t_2) \beps_2, t_2, 0; S_{\theta})\|_2^2 \\
    & \leq e^{2L_s t_2} \| \be + \sigma(t_1) \beps_1 - \sigma(t_2) \beps_2\|_2^2 \\
    & \leq e^{2L_s t_2} [\| \be \|_2^2 + (\sigma^2(t_1) + \sigma^2(t_2)) \| \beps \|_2^2]. \label{eqn:first_error}
\end{align}
Using the Lemma 1, we can bound the second term in \eqref{eqn:total_error} by:
\begin{align}
    \|\hat{\mathbf{u}}^h(t_2, t_2)  - \buth \|_2^2 & \leq e^{2 L_s t_2}\sigma^2(t) \delta.\label{eqn:second_error}
\end{align}
Combining \eqref{eqn:first_error} and \eqref{eqn:second_error}, we get
\[
\| \buhat(t_1, t_2) - \buth \|_2^2 \leq e^{2L_s t_2} [\| \be \|_2^2 + \sigma^2(t)\delta + (\sigma^2(t_1) + \sigma^2(t_2)) \| \beps \|_2^2]
\]
Since $\| \beps \|^2_2 \sim \chi^2(d)$, from the Lemma 1 in \cite{laurent2000adaptive}, we have
\[
p(\| \beps \|^2_2 \geq d + 2 \sqrt{d q} + 2q) \leq e^{-q}
\]
Let $e^{-q} = \lambda$, we get 
\[
p(\| \beps \|^2_2 \geq d + 2 \sqrt{d \log\frac{1}{\lambda}} + 2\log\frac{1}{\lambda}) \leq \lambda
\]
we then have with probability at least $1-\lambda$ such that
\[
\| \buhat(t_1, t_2) - \buth \|_2^2 \leq e^{2L_s t_2} [\| \be \|_2^2 + \sigma^2(t_2)\delta + (\sigma^2(t_1) + \sigma^2(t_2)) C_{\lambda}].
\]
where 
\[
c_\lambda = d + 2 \sqrt{d \log\frac{1}{\lambda}} + 2\log\frac{1}{\lambda}
\]
\end{proof}

\section{Metrics}\label{sec:metrics}
We define the following metrics that measure the distance between two empirical distributions $\{ \bu_i \}_{i=1}^N$ and $\{ \bv_j \}_{j=1}^M$:
	 \begin{itemize}
     \item Total Variation Distance (TVD) (two distributions have same size $N=M$): 
    \begin{equation}\label{eqn:tvd}
        \text{TVD}(\bu,\bv) = \frac{1}{N} \sum_{i=1}^N \frac{\| \bu_i - \bv_i \|_1}{\| \bv_i \|_1}
    \end{equation}
	 \item Relative root mean squared error (RMSE) (two distributions have same size $N=M$):
        \begin{equation}\label{eqn:RMSE}
            \text{RMSE}(\bu, \bv) = \frac{1}{N} \sum_{i=1}^N \frac{\| \bu_i - \bv_i \|_2}{\| \bv_i \|_2}
        \end{equation}
    \item Maximum Mean Discrepancy (MMD):
    \begin{equation}\label{eqn:mmd}
    \text{MMD}(\bu,\bv) = 
    \frac{1}{N^2} \sum_{i=1}^N \sum_{j=1}^N k(\bu_i, \bu_j) + \frac{1}{M^2} \sum_{i=1}^M \sum_{j=1}^M k(\bv_i, \bv_j) - \frac{2}{NM} \sum_{i=1}^N \sum_{j=1}^M k(\bu_i, \bv_j)
    \end{equation}
    where
    \[
    k(\bu, \bv) = \exp\left(-\frac{\|\bu - \bv\|^2}{2l^2}\right),
    \]
    is a Gaussian Kernel, and we let $l=0.01$ in this paper.
    
    \item Wasserstein-2 distance ($\mWa$):
    \begin{equation}\label{eqn:w2}
    \mWa(\bu,\bv) = \sqrt{\min_{\pi}\sum_{i=1}^N \sum_{j=1}^M \pi_{ij} \| \bu_i - \bv_j \|_2^2},
    \end{equation}
    where $\pi_{ij}$ is the optimal transport plan that satisfies the constraints:
    \[
    \sum_{j=1}^M \pi_{ij} = \frac{1}{N}, ~ \forall i; \quad \text{and} \quad \sum_{i=1}^N \pi_{ij} = \frac{1}{M}, ~ \forall j.
    \]
    In practice, we use the POT package \cite{flamary2021pot} for computation.
	 \end{itemize}

\section{Supplementary Algorithm}\label{sec:supp_alg}
For completeness, we provide Algorithm~\ref{alg:t} for determining the optimal $t$ in BPD.
\begin{algorithm}
\caption{Selection of $t$}
\label{alg:t}
\begin{algorithmic}[1]
\REQUIRE Two datasets $\{ \bul_j \}_{j=1}^M$ and $\{ \buth_i \}_{i=1}^N$, terminal searching time $T_e$, number of steps $N_{t}$, and a metric $\mathcal{M}$.
\STATE Initialize $d_{\min} \gets \infty$
\STATE Initialize $t^{*} \gets 0$, 
\FOR{$p \gets 0$ to $N_{t}-1$}
    \STATE $t_1 \gets p \cdot \frac{T_e}{N_{t}-1}$
    \STATE Compute $d = \mathcal{M}(p(\bul(t)), p(\buth(t)))$
    \IF{$d < \min$}
        \STATE $\min \gets d$
        \STATE $t^* \gets t$
    \ENDIF
\ENDFOR
\RETURN Optimal $t^*$
\end{algorithmic}
\end{algorithm}

\begin{algorithm}[H]
\caption{Conditional Diffusion Model for Super-resolution}
\label{alg:uh}
\begin{algorithmic}[1]
\REQUIRE Paired dataset $\mathcal{T} = \{ \buth_i, \buh_i \}_{i=1}^N$, noise scheduling function $\sigma(t)$, batch size $B$ and max iteration $Iter$.
\STATE Initialize $k=0$.
\WHILE{$k<Iter$}
    \STATE Sample $\{ \buth_j, \buh_j \}_{j=1}^B \sim \mathcal{T}$ 
    \STATE $t \sim U[0, 1]$
    \STATE $\beps_j \sim \mN(\mz, \mI)$ for $j=1, \cdots, B$
    \STATE Compute $\buh_j(t) = \buh_j + \sigma(t) \beps_j$
    \STATE  Update $\xi$ using Adam to minimize the empirical loss \eqref{eqn:loss_sr}
    \STATE $k \gets k+1$
\ENDWHILE
\RETURN Conditional diffusion model $S_{\xi}(\bu^h(t), \buth, t)$
\end{algorithmic}
\end{algorithm}

\section{Supplementary Figures}\label{sec:supp}
The results for Pink and Brown noises of the climate example in Section~\ref{sec:climate} are presented in Figure~\ref{fig:CM_2} and Figure~\ref{fig:CM_3} respectively.

\begin{figure}[h]
    \centering
    \includegraphics[width=0.55\linewidth]{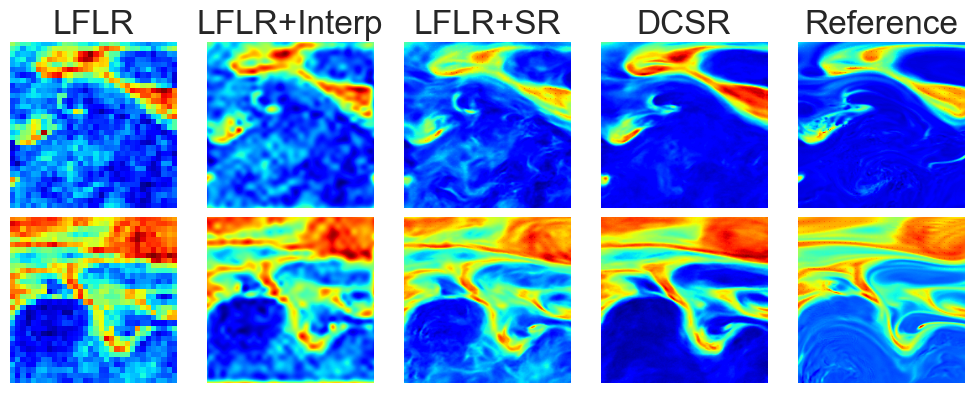}
    \includegraphics[width=0.38\linewidth]{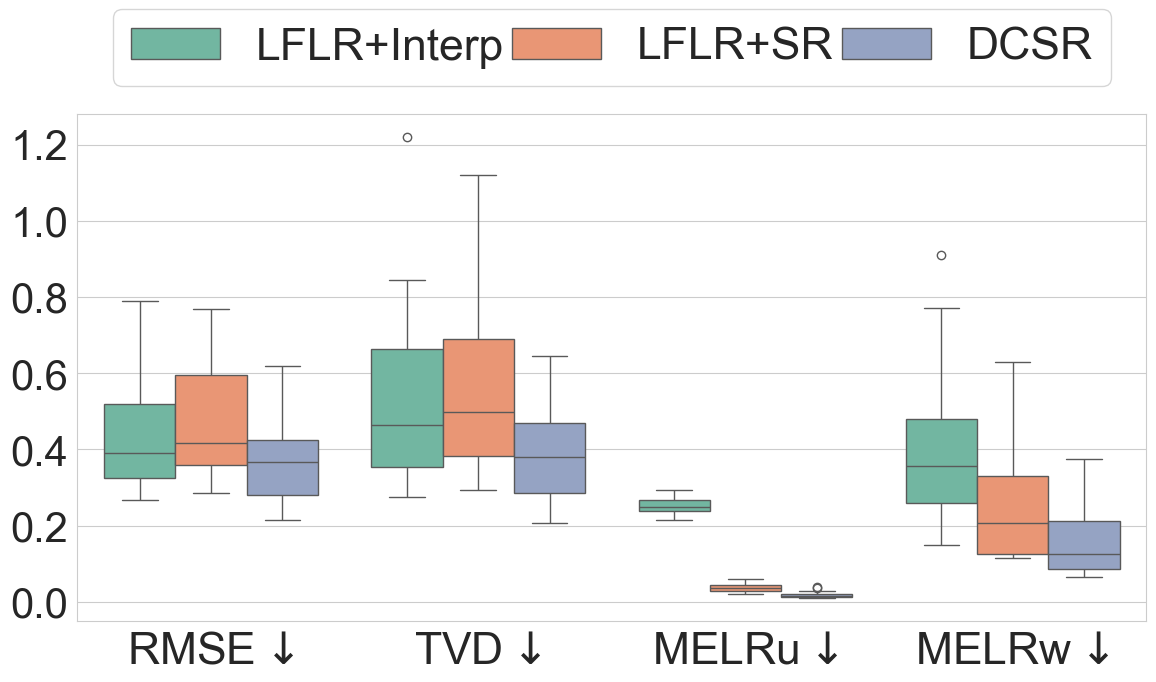}
    
    \includegraphics[width=0.55\linewidth]{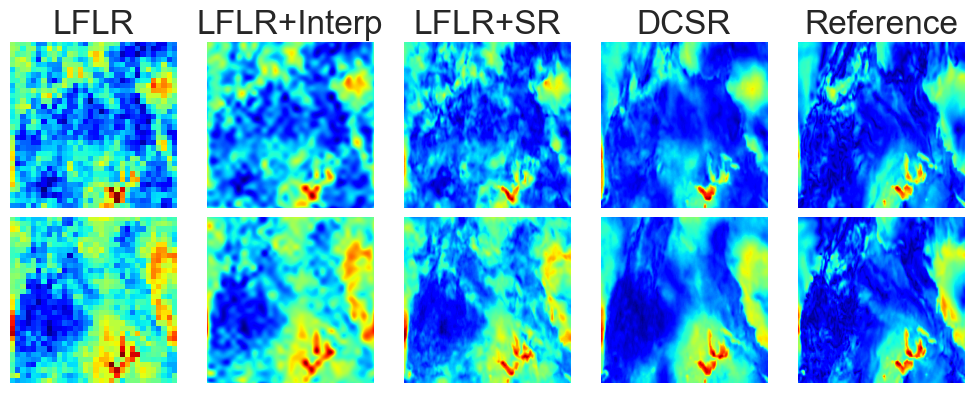}
    \includegraphics[width=0.38\linewidth]{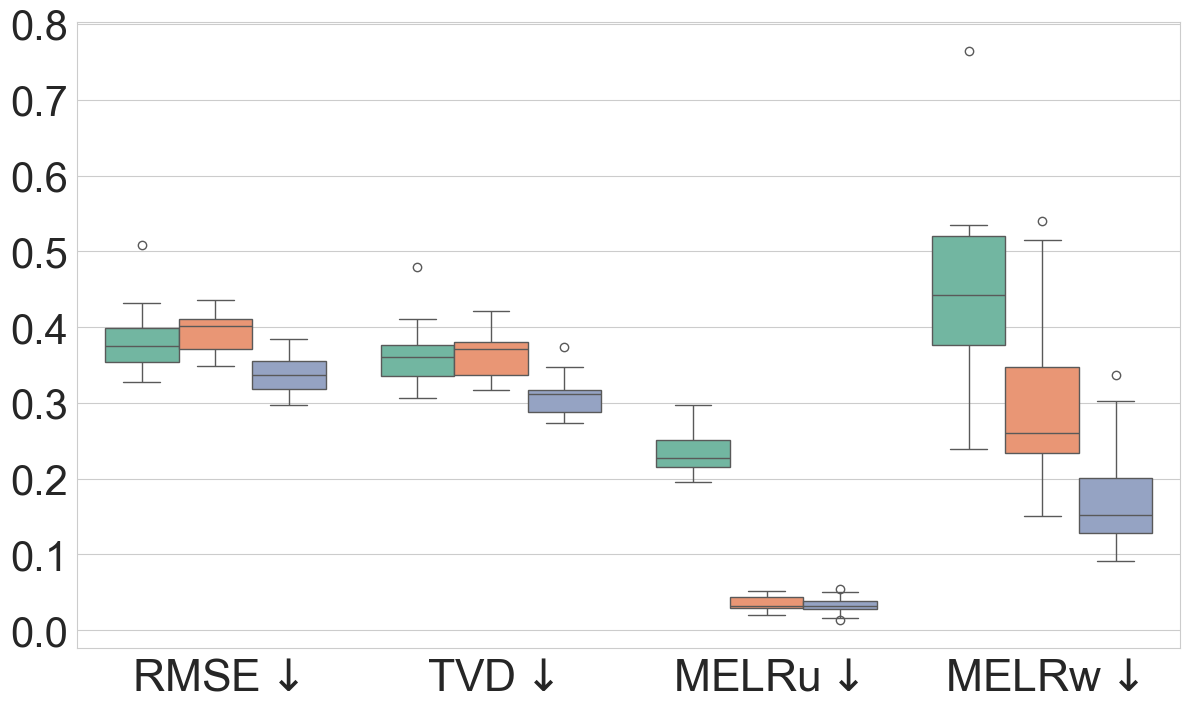}
    
    \caption{Top row presents results of low resolution potential vorticity data polluted with Pink noise and the Bottom shows the results of low resolution wind speed polluted with Pink noise. The left columns compare low-fidelity data with results from various refinement methods. The right columns present boxplots comparing performance of various refinement methods across metrics.}
    \label{fig:CM_2}
\end{figure}

\begin{figure}[h]
    \centering
    \includegraphics[width=0.55\linewidth]{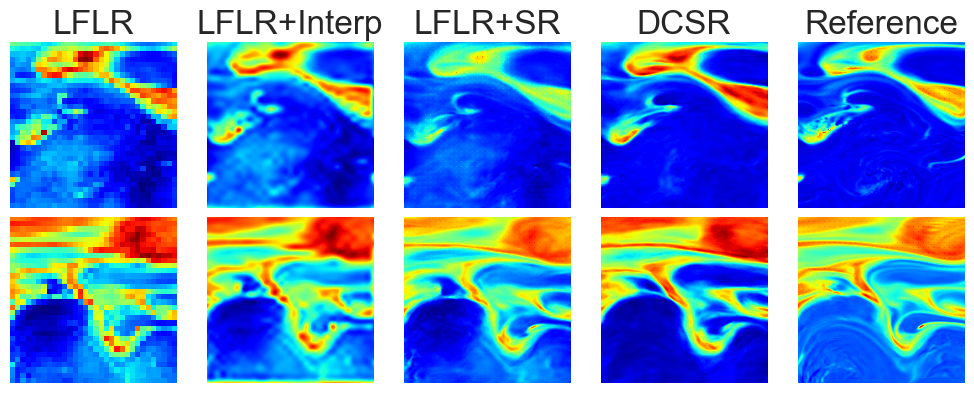}
    \includegraphics[width=0.38\linewidth]{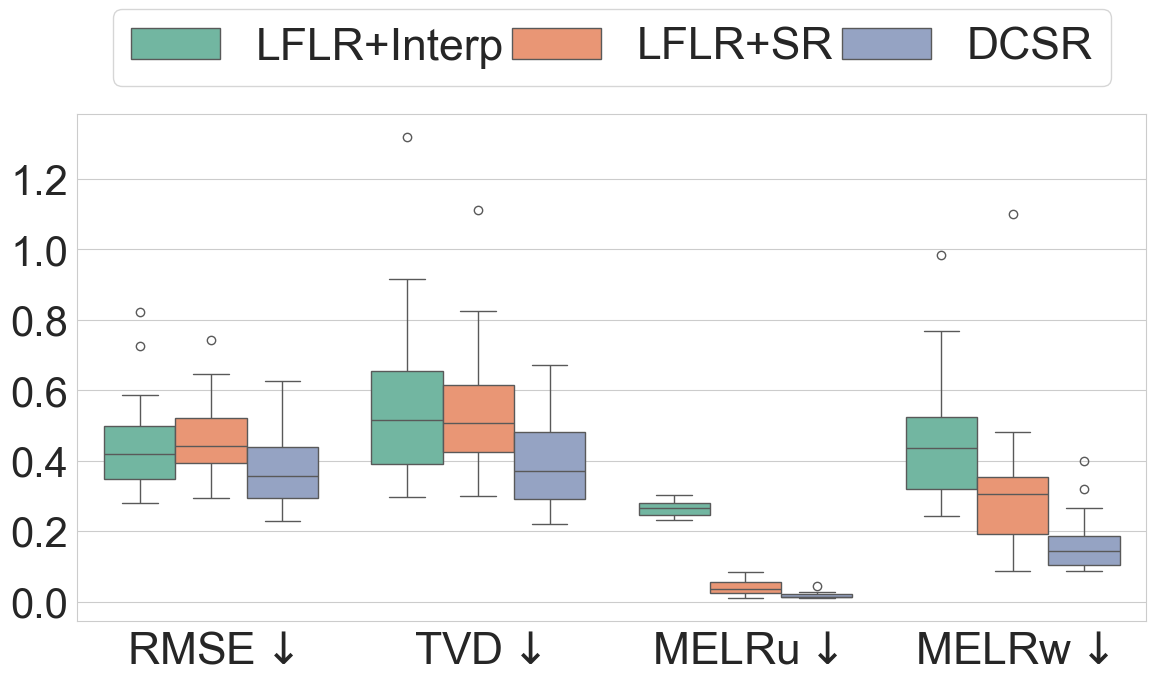}
    
    \includegraphics[width=0.55\linewidth]{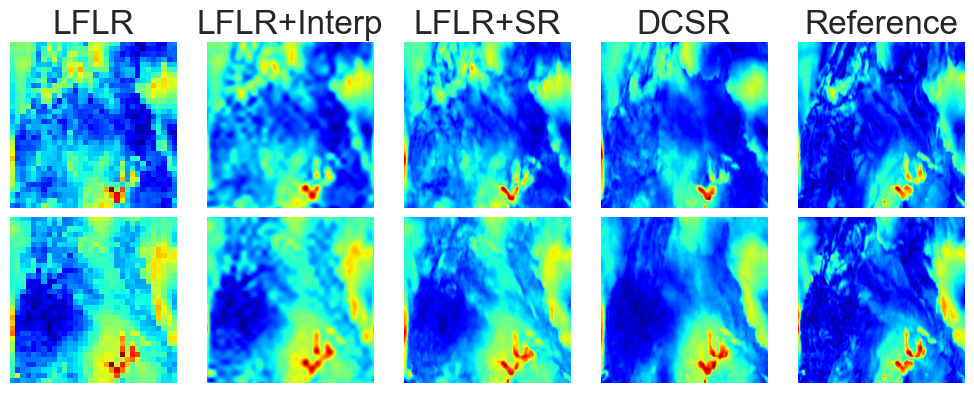}
    \includegraphics[width=0.38\linewidth]{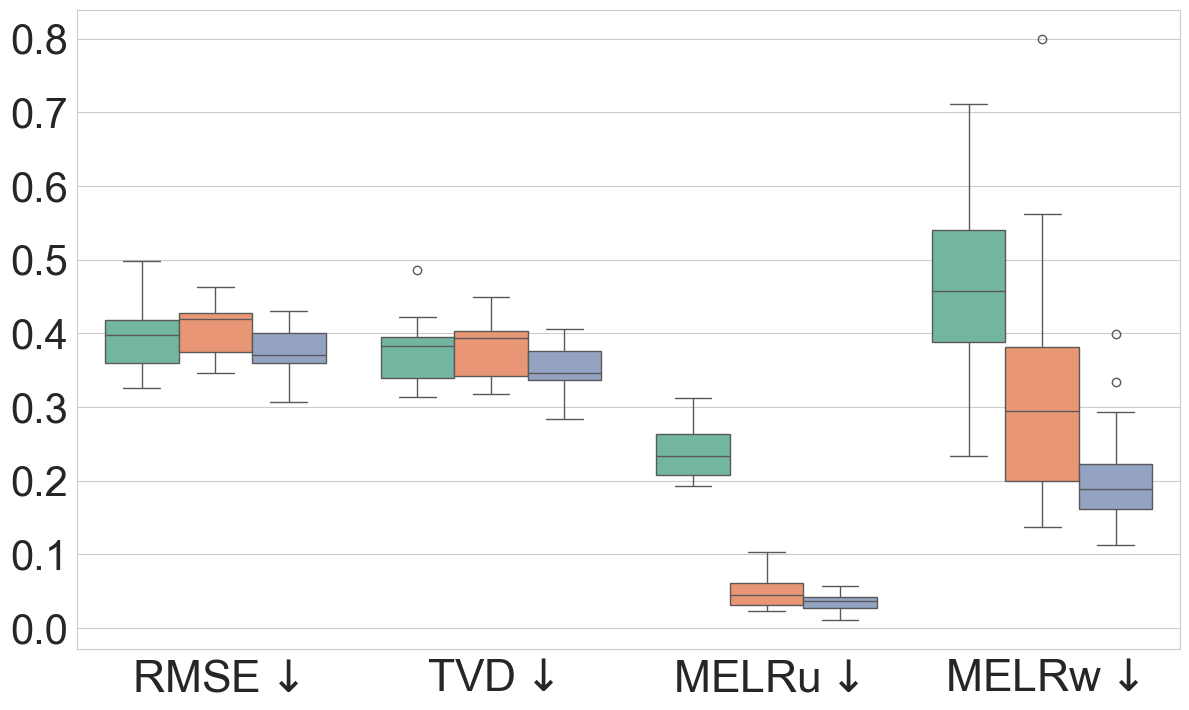}
    
    \caption{Top row presents results of low resolution potential vorticity data polluted with Brown noise and the Bottom shows the results of low resolution wind speed polluted with Brown noise. The left columns compare low-fidelity data with results from various refinement methods. The right columns present boxplots comparing performance of various refinement methods across metrics.}
    \label{fig:CM_3}
\end{figure}

\end{document}